\documentclass[10pt,reqno]{amsart}
\usepackage[hypertex]{hyperref}

\usepackage{amsmath,amsthm}
\usepackage{amssymb}
\usepackage{euscript}
\usepackage[all,tips]{xy}

\numberwithin{equation}{subsection}

\allowdisplaybreaks[4]

\newcommand{\sqsp}{\renewcommand{\baselinestretch}{1.17}\tiny\normalsize}

\newtheorem{theorem}[subsection]{Theorem}

\newtheorem{proposition}[subsection]{Proposition}
\newtheorem{corollary}[subsection]{Corollary}
\newtheorem{twist}[subsection]{Twisting Principles}

\theoremstyle{definition}
\newtheorem{definition}[subsection]{Definition}
\newtheorem{example}[subsection]{Example}
\newtheorem{remark}[subsection]{Remark}

\newcommand{\fg}{\mathfrak{g}}

\newcommand{\sltwo}{\mathfrak{sl}_2}
\newcommand{\sln}{\mathfrak{sl}_n}
\newcommand{\slnplusone}{\mathfrak{sl}_{n+1}}

\newcommand{\bC}{\mathbf{C}}
\newcommand{\bk}{\mathbf{k}}
\newcommand{\bZ}{\mathbf{Z}}

\newcommand{\al}{{\alpha_\lambda}}
\newcommand{\mualpha}{\mu_\alpha}

\newcommand{\rhoalpha}{\rho_\alpha}

\newcommand{\qp}{\mathbb{A}^{2|0}_q}
\newcommand{\qpalpha}{\mathbb{A}^{2|0}_{q,\alpha}}

\newcommand{\qn}{\mathbb{A}^{n|0}_q}
\newcommand{\qnalpha}{\mathbb{A}^{n|0}_{q,\alpha}}
\newcommand{\qfn}{\mathbb{A}^{0|n}_q}
\newcommand{\qfnalpha}{\mathbb{A}^{0|n}_{q,\alpha}}

\newcommand{\uhg}{U_h(\fg)}
\newcommand{\uq}{U_q(\sltwo)}
\newcommand{\uqalpha}{\uq_\alpha}
\newcommand{\uqbeta}{\uq_\beta}
\newcommand{\uqg}{U_q(\fg)}

\newcommand{\uqsln}{U_q(\sln)}

\newcommand{\mqeta}{M_q(\eta)}

\newcommand{\ch}{\bC[[h]]}
\newcommand{\ax}{\alpha_\xi}
\newcommand{\ir}{{\begin{bmatrix}n \\ r\end{bmatrix}_q}}

\newcommand{\dqx}{\partial_{q,x}}
\newcommand{\dqy}{\partial_{q,y}}
\newcommand{\sigmax}{\sigma_x}
\newcommand{\sigmay}{\sigma_y}

\newcommand{\rk}{R^{\alpha^k}}

\newcommand{\nicearrow}{\SelectTips{cm}{10}}

\DeclareMathOperator{\Hom}{Hom}

\DeclareMathOperator{\diag}{diag}

\begin{document}

\title[Hom-quantum groups III]{Hom-quantum groups III: Representations and module Hom-algebras}
\author{Donald Yau}

\begin{abstract}
We study Hom-quantum groups, their representations, and module Hom-algebras.  Two Twisting Principles for Hom-type algebras are formulated, and construction results are proved following these Twisting Principles.  Examples include Hom-quantum $n$-spaces, Hom-quantum enveloping algebras of Kac-Moody algebras, Hom-Verma modules, and Hom-type analogs of $\uq$-module-algebra structures on the quantum planes.
\end{abstract}



\address{Department of Mathematics\\
    The Ohio State University at Newark\\
    1179 University Drive\\
    Newark, OH 43055, USA}
\email{dyau@math.ohio-state.edu}

\date{\today}
\maketitle

\sqsp

\section{Introduction}

This paper is part of an on-going effort \cite{yau6} - \cite{yau10} to study twisted, Hom-type generalizations of quantum groups, the Yang-Baxter equations (YBEs) \cite{baxter,baxter2,perk,skl1,skl2,yang}, and related algebraic structures.  We use the name \emph{Hom-quantum groups} colloquially to refer to Hom-type generalizations of quantum groups, i.e., bialgebras, possibly with additional structures.  Roughly speaking, quantum group theory involves the study of bialgebras and Hopf algebras that are not-necessarily commutative or cocommutative.  Hom-quantum groups extend quantum groups by introducing non-(co)associativity, and the level of non-(co)associativity is controlled by a certain twisting map $\alpha$ ($=$ the ``Hom" in Hom-quantum groups).  For example, the Hom-quantum enveloping algebras (Examples \ref{ex:uqg} and \ref{ex:kacmoody}) and the FRT Hom-quantum groups (section 4 in \cite{yau10}) are all simultaneously non-associative, non-coassociative, non-commutative, and non-cocommutative.

Some important properties of quantum groups have been shown to have counterparts in the setting of Hom-quantum groups.  For example, it is well-known that each module over a quasi-triangular bialgebra \cite{dri87} has a canonical solution of the YBE.  In the world of Hom-quantum groups, for a quasi-triangular Hom-bialgebra whose Hom-braiding element is $\alpha$-invariant, there is a canonical solution of the Hom-Yang-Baxter equation (HYBE) \cite{yau6,yau7} associated to each module \cite{yau9}.  There is a similar statement for the dual objects of comodules over suitable cobraided Hom-bialgebras \cite{yau10}.  Likewise, quantum geometry (in particular, quantum group coactions on the quantum planes) has been generalized to what we call Hom-quantum geometry, which involves Hom-quantum group coactions on the Hom-quantum planes (section 7 in \cite{yau10}).

The purpose of this paper is to advance the study of Hom-quantum groups (that is, Hom-bialgebras, possibly with additional structures) and their representations, as initiated in \cite{yau9,yau10}.  We also consider module Hom-algebras and Hom-quantum geometry on the Hom-quantum planes.  This paper subsumes the earlier preprint \cite{yau4}.  Below is a description of the rest of this paper.

In section \ref{sec:homalgebras} we formulate the Twisting Principles \ref{twistingprinciple} for Hom-type algebras, which we will use as a guide throughout the rest of this paper.  The first Twisting Principle says that an ordinary algebraic structure can be twisted into a corresponding type of Hom-algebraic structure via suitable endomorphisms.  The first example of the first Twisting Principle was considered by the author in \cite{yau2}.  It is by now ubiquitous in the study of Hom-type objects; see the discussion after the Twisting Principles.  The second Twisting Principle says that a Hom-type algebra, without any additional data, can be twisted along its own twisting map to give a derived Hom-type algebra.  Moreover, the second Twisting Principle can often be applied repeatedly, so a Hom-type algebra gives rise to a derived sequence of Hom-type algebras.  We prove instances of the Twisting Principles \ref{twistingprinciple} for Hom-(co)associative (co)algebras and Hom-bialgebras (Theorems \ref{thm:hombialg} and \ref{thm:secondtp}).  As examples, we construct multi-parameter classes of (fermionic) Hom-quantum $n$-spaces (Examples \ref{ex:hqspace} and \ref{ex:fhqspace}) and of Hom-quantum enveloping algebras of complex semisimple Lie algebras and of Kac-Moody algebras (Examples  \ref{ex:uqg} and \ref{ex:kacmoody}).  These Hom-quantum enveloping algebras are all simultaneously non-associative, non-coassociative, non-commutative, and non-cocommutative.

In section \ref{sec:qthb} we consider quasi-triangular and cobraided Hom-bialgebras.   These Hom-quantum groups were introduced in \cite{yau9,yau10}, where a version of the first Twisting Principle was proved in each case.  Quasi-triangular and cobraided Hom-bialgebras are the Hom-type analogs of Drinfel'd's quasi-triangular bialgebras \cite{dri87} and of cobraided bialgebras \cite{hay,lt,majid91,sch}.  As shown in \cite{yau9,yau10}, modules over quasi-triangular Hom-bialgebras and comodules over cobraided Hom-bialgebras are related to solutions of the HYBE \cite{yau6,yau7}.  We establish the second Twisting Principle for quasi-triangular and cobraided Hom-bialgebras (Theorems \ref{thm:qtsecond} and \ref{thm:cobraidedsecond}).  In particular, a quasi-triangular Hom-bialgebra gives rise to a derived sequence (a double-sequence if the twisting map is surjective) of quasi-triangular Hom-bialgebras.  Likewise, a cobraided Hom-bialgebra gives rise to a derived sequence (a double-sequence if the twisting map is injective) of cobraided Hom-bialgebras.

In section \ref{sec:modules} we study modules over Hom-(co)associative (co)algebras.  We establish the second Twisting Principle \ref{twistingprinciple} for modules over a Hom-associative algebra (Corollary \ref{cor:twisthommodule}).  It says that each module over a Hom-associative algebra gives rise to a derived double-sequence of modules over derived Hom-associative algebras.   For the first Twisting Principle, we show that given a module over an associative algebra, each suitable pair of morphisms gives rise to a derived double-sequence of modules over Hom-associative algebras (Theorem \ref{thm:twistmodule}).  As examples, we construct multi-parameter classes of finite-dimensional modules over the Hom-quantum enveloping algebras $\uqsln_\alpha$ of $\sln$ (Examples \ref{ex:uqmodule} and \ref{ex:slnmodule}).  We also construct a multi-parameter class of infinite-dimensional Hom-Verma modules over the Hom-quantum enveloping algebra $\uq_\alpha$ of $\sltwo$ (Example \ref{ex:verma}).  The Twisting Principles for comodules over Hom-coassociative coalgebras are stated at the end of section \ref{sec:modules} (Theorems \ref{thm:twisthomco} and \ref{thm:twistco}).

In section \ref{sec:mha} we study module Hom-algebras, which are the Hom-type analogs of module-algebras.  If $H$ is a Hom-bialgebra, then an $H$-module Hom-algebra is a Hom-associative algebra $A$ together with an $H$-module structure, satisfying the module Hom-algebra axiom \eqref{modhomalg}.  We give an alternative characterization of the module Hom-algebra axiom in terms of the multiplication on $A$ (Theorem \ref{thm:mhachar}).  Then we establish the Twisting Principles for module Hom-algebras (Theorems \ref{thm:twistmha} and \ref{thm:twistma}).  We use the results on module Hom-algebras to study Hom-quantum geometry on the Hom-quantum planes.  In Example \ref{ex:hqg} we construct a multi-parameter family of $\uqalpha$-module Hom-algebra structures on the Hom-quantum planes, generalizing the well-known $\uq$-module-algebra structure on the quantum plane \cite{msmith} that is defined in terms of the quantum partial derivatives \eqref{qpartial}.  In Example \ref{ex2:hqg} we construct Hom-type analogs of a non-standard $\uq$-module-algebra structure on the quantum plane \cite{ds}.

\section{The Twisting Principles and Hom-quantum groups}
\label{sec:homalgebras}

In this section, we construct some specific examples of Hom-quantum groups.  Although our main interest is in Hom-quantum groups, our construction results work more generally.  We state the Twisting Principles \ref{twistingprinciple}, which we use as a guide throughout the rest of this paper.  As examples of the Twisting Principles, we construct classes of Hom-(co)associative (co)algebras and Hom-bialgebras (Theorems \ref{thm:hombialg} and \ref{thm:secondtp}).

To illustrate these results, we provide several examples related to the quantum $n$-spaces and the quantum enveloping algebras.  In Example \ref{ex:hqspace}, we construct a multi-parameter class of Hom-quantum $n$-spaces $\qnalpha$, which are the Hom-type analogs of the quantum $n$-spaces $\qn$.  We also consider the fermionic version of the Hom-quantum $n$-spaces (Example \ref{ex:fhqspace}).  Next, we construct a multi-parameter, uncountable family of Hom-bialgebra deformations of the Drinfel'd-Jimbo quantum group $\uqg$ (Example \ref{ex:uqg}), where $\fg$ is any complex semisimple Lie algebra of type $A$, $D$, or $E$.  More generally, we construct a multi-parameter, uncountable family of Hom-bialgebra deformations of the quantum Kac-Moody algebra $\uhg$, where $\fg$ is any symmetrizable Kac-Moody algebra (Example \ref{ex:kacmoody}).

\subsection{Conventions and notations}

We work over a fixed associative and commutative ring $\bk$ of characteristic $0$.  Modules, tensor products, and linear maps are all taken over $\bk$.  If $V$ and $W$ are $\bk$-modules, then $\tau = \tau_{V,W} \colon V \otimes W \to W \otimes V$ denotes the twist isomorphism, $\tau(v \otimes w) = w \otimes v$.

Given a bilinear map $\mu \colon V^{\otimes 2} \to V$ and elements $x,y \in V$, we often write $\mu(x,y)$ as $xy$.  For a map $\Delta \colon V \to V^{\otimes 2}$, we use Sweedler's notation \cite{sweedler} for comultiplication: $\Delta(x) = \sum_{(x)} x_1 \otimes x_2$.

If $\alpha \colon V \to V$ is a self-map of a module $V$ and $n \geq 0$, then $\alpha^n$ denotes the $n$-fold composition $\alpha \circ \cdots \circ \alpha$ of copies of $\alpha$, with $\alpha^0 \equiv Id_V$.

\subsection{A brief history of Hom-type algebras}
Hom-type algebras first appeared in the form of Hom-Lie algebras \cite{hls}, which satisfy an $\alpha$-twisted version of the Jacobi identity.  Hom-Lie algebras are closely related to deformed vector fields \cite{ama,hls,ls,ls2,ls3,rs,ss} and number theory \cite{larsson}.  Hom-associative algebras were introduced in \cite{ms} to construct Hom-Lie algebras using the commutator bracket.  The universal Hom-associative algebra of a Hom-Lie algebra was studied in \cite{yau}.  It was further shown in \cite{yau3} that a unital version of the universal Hom-associative algebra has the structure of a Hom-bialgebra.  Variations of Hom-Lie and Hom-associative algebras were studied in \cite{fg}.  Some classification results about Hom-associative algebras can be found in \cite{fg2,gohr}.  The authors of \cite{cg} constructed Hom-(co)associative (co)algebras and Hom-bialgebras with bijective twisting maps as (co)algebras and bialgebras in a certain tensor category.  Other papers about Hom-type structures are \cite{ams}, \cite{hb}, \cite{jl}, \cite{mak} - \cite{ms4}, \cite{yau2}, and \cite{yau4} - \cite{yau10}.

\begin{definition}
\label{def:homas}
\begin{enumerate}
\item
A \textbf{Hom-associative algebra} \cite{ms} $(A,\mu,\alpha)$ consists of a  $\bk$-module $A$, a bilinear map $\mu \colon A^{\otimes 2} \to A$ (the multiplication), and a linear self-map $\alpha \colon A \to A$ (the twisting map) such that:
\begin{equation}
\label{homassaxioms}
\begin{split}
\alpha \circ \mu &= \mu \circ \alpha^{\otimes 2} \quad \text{(multiplicativity)},\\
\mu \circ (\alpha \otimes \mu) &= \mu \circ (\mu \otimes \alpha) \quad \text{(Hom-associativity)}.
\end{split}
\end{equation}
A morphism of Hom-associative algebras is a linear map of the underlying $\bk$-modules that commutes with the twisting maps and the multiplications.
\item
A \textbf{Hom-coassociative coalgebra} \cite{ms2,ms4} $(C,\Delta,\alpha)$ consists of a $\bk$-module $C$, a linear map $\Delta \colon C \to C^{\otimes 2}$ (the comultiplication), and a linear self-map $\alpha \colon C \to C$ (the twisting map) such that:
\begin{equation}
\label{homcoassaxioms}
\begin{split}
\alpha^{\otimes 2} \circ \Delta &= \Delta \circ \alpha \quad \text{(comultiplicativity)},\\
(\alpha \otimes \Delta) \circ \Delta &= (\Delta \otimes \alpha) \circ \Delta \quad \text{(Hom-coassociativity)}.
\end{split}
\end{equation}
A morphism of Hom-coassociative coalgebras is a linear map of the underlying $\bk$-modules that commutes with the twisting maps and the comultiplications.
\item
A \textbf{Hom-bialgebra} \cite{ms2,yau3} is a quadruple $(H,\mu,\Delta,\alpha)$ in which $(H,\mu,\alpha)$ is a Hom-associative algebra, $(H,\Delta,\alpha)$ is a Hom-coassociative coalgebra, and the following condition holds:
\begin{equation}
\label{def:hombi}
\Delta \circ \mu = \mu^{\otimes 2} \circ (Id \otimes \tau \otimes Id) \circ \Delta^{\otimes 2}.
\end{equation}
A morphism of Hom-bialgebras is a linear map of the underlying $\bk$-modules that commutes with the twisting maps, the multiplications, and the comultiplications.
\end{enumerate}
\end{definition}

In terms of elements, \eqref{def:hombi} means that
\[
\Delta(ab) = \sum_{(a)(b)} a_1b_1 \otimes a_2b_2
\]
for all $a, b \in H$.  This compatibility condition also holds in any bialgebra.  When there is no danger of confusion, we will denote a Hom-associative algebra $(A,\mu,\alpha)$ simply by $A$.  The same remark applies to other Hom-type objects.

\begin{remark}
In our definitions of a Hom-(co)associative (co)algebra and of a Hom-bialgebra, we require that the twisting map $\alpha$ be (co)multiplicative.  The original definition of a Hom-associative algebra in \cite{ms} does not have this multiplicativity assumption.  We require multiplicativity in our Hom-type algebras because in all of our examples, the twisting maps are already multiplicative.  Also the (co)multiplicativity condition is necessary for the second Twisting Principle for Hom-(co)associative (co)algebras and Hom-bialgebras (Theorem \ref{thm:secondtp}).
\end{remark}

Observe that a Hom-bialgebra is neither associative nor coassociative, unless of course $\alpha = Id$, in which case we have a bialgebra.  Instead of (co)associativity, in a Hom-bialgebra we have Hom-(co)associativity, $\mu \circ (\alpha \otimes \mu) = \mu \circ (\mu \otimes \alpha)$ and $(\alpha \otimes \Delta) \circ \Delta = (\Delta \otimes \alpha) \circ \Delta$.  So, roughly speaking, the degree of non-(co)associativity in a Hom-bialgebra is measured by how far the twisting map $\alpha$ deviates from the identity map.

One can think of a Hom-type algebra as an algebraic structure equipped with a twisting map $\alpha$, satisfying an $\alpha$-twisted version of the usual axioms.  Then it makes sense that an ordinary algebra can be twisted into a Hom-type algebra via a suitable self-map.  Moreover, one should be able to twist a Hom-type algebra along its own twisting map to obtain another Hom-type algebra.  We state these ideas as follows.

\begin{twist}
\label{twistingprinciple}
\begin{enumerate}
\item
If $A$ is an algebraic structure and $\alpha$ is a suitable self-map of $A$, then one obtains a corresponding type of Hom-algebraic structure $A_\alpha$ by twisting the structure maps of $A$ by $\alpha$.
\item
If $(A,\alpha)$ is a Hom-algebraic structure, then there is a derived Hom-algebraic structure $A^1$, which is obtained from $A$ by twisting its structure maps by $\alpha$.
\end{enumerate}
\end{twist}

We refer to these statements as the first and the second Twisting Principles. In practice, one has to be careful about the definition of a ``self-map" and about how and which parts of the structures of $A$ are twisted by $\alpha$, especially in the second Twisting Principle.  We will make precise both Twisting Principles for various algebraic structures.  The Twisting Principles are our main tool in constructing examples of Hom-type objects.

The first Twisting Principle allows one to construct many examples of Hom-type objects from ordinary ones and suitable endomorphisms.  It is a standard tool in the study of Hom-type objects.  In fact, most concrete examples of Hom-type algebras in the literature are constructed using the first Twisting Principle.  It was introduced by the author in \cite{yau2} (Theorem 2.3), where it is shown, in particular, that associative and Lie algebras can be twisted into Hom-associative and Hom-Lie algebras via algebra endomorphisms.  It has since been employed by various authors.  See, for example, \cite{ama} (Theorem 2.7), \cite{ams} (Theorems 1.7 and 2.6), \cite{fg2} (Section 2), \cite{gohr} (Proposition 1), \cite{mak} (Theorems 2.1 and 3.5), \cite{ms4} (Theorem 3.15 and Proposition 3.30), and \cite{yau3} - \cite{yau10}.

Unlike the first Twisting Principle, the second Twisting Principle says that a Hom-type structure can be twisted into another Hom-type structure without any additional data.  Moreover, it can be iterated ad infinitum.  In fact, it can be applied to the derived Hom-type object $A^1$ to obtain a second derived Hom-type object $A^2$, and so forth.  We will see several explicit cases of the second Twisting Principle later in this paper.

The following result illustrates the first Twisting Principle for (co)associative algebra and bialgebras.

\begin{theorem}[\cite{ms4,yau2,yau3}]
\label{thm:hombialg}
Let $(A,\mu)$ be an associative algebra, $(C,\Delta)$ be a coassociative coalgebra, and $(H,\mu,\Delta)$ be a bialgebra.
\begin{enumerate}
\item
If $\alpha \colon A \to A$ is an algebra morphism, then $A_\alpha = (A,\mu_\alpha,\alpha)$ is a Hom-associative algebra, where $\mu_\alpha = \alpha \circ \mu$.
\item
If $\alpha \colon C \to C$ is a coalgebra morphism, then $C_\alpha = (C,\Delta_\alpha,\alpha)$ is a Hom-coassociative coalgebra, where $\Delta_\alpha = \Delta \circ \alpha$.
\item
If $\alpha \colon H \to H$ is a bialgebra morphism, then $H_\alpha = (H,\mu_\alpha,\Delta_\alpha,\alpha)$ is a Hom-bialgebra, where $\mu_\alpha = \alpha \circ \mu$ and $\Delta_\alpha = \Delta \circ \alpha$.
\end{enumerate}
\end{theorem}

\begin{proof}
As mentioned above, it was proved in \cite{yau2} (Theorem 2.3) that $A_\alpha = (A,\mu_\alpha,\alpha)$ is a Hom-associative algebra.  Indeed, the Hom-associativity axiom $\mu_\alpha \circ (\alpha \otimes \mu_\alpha) = \mu_\alpha \circ (\mu_\alpha \otimes \alpha)$ is equal to $\alpha^2$ applied to the associativity axiom of $\mu$.  Likewise, both sides of the multiplicativity axiom $\alpha \circ \mu_\alpha = \mu_\alpha \circ \alpha^{\otimes 2}$ are equal to $\alpha^2 \circ \mu$.  Dualizing the previous arguments \cite{ms4}, one checks that $C_\alpha = (C,\Delta_\alpha,\alpha)$ is a Hom-coassociative coalgebra.  For the last assertion, the compatibility condition \eqref{def:hombi} for $\Delta_\alpha = \Delta \circ \alpha$ and $\mu_\alpha = \alpha \circ \mu$ is equal to $\alpha^2 \otimes \alpha^2$ followed by \eqref{def:hombi}, which holds in any bialgebra.
\end{proof}

Using Theorem \ref{thm:hombialg} we now discuss the (fermionic) Hom-quantum $n$-spaces and the Hom-quantum enveloping algebras of semisimple Lie algebras and, more generally, Kac-Moody algebras.

\begin{example}[\textbf{Hom-quantum $n$-spaces}]
\label{ex:hqspace}
In this example, we construct Hom versions of quantum $n$-spaces, which are themselves quantum analogs of affine spaces.  Fix a non-zero scalar $q \in \bk \setminus \{\pm 1\}$ and an extended integer $n \in \{2,3,\ldots,\infty\}$.  Suppose $x_i$ for $1 \leq i \leq n$ (or $x_i$ for all $i \geq 1$ if $n=\infty$) are independent variables.  Let $\bk\{x_r\}_{r=1}^n$ be the unital associative algebra generated by these variables.  The \textbf{quantum $n$-space} is defined as the quotient
\[
\qn = \frac{\bk\{x_r\}_{r=1}^n}{(x_jx_i - qx_ix_j \text{ for $i<j$})}.
\]
The relation
\[
x_jx_i = qx_ix_j \quad\text{for $i<j$}
\]
is called the \emph{quantum commutation relation}.  In particular, $\qn$ is non-commutative.  When $n=2$, we use $x=x_1$ and $y=x_2$, and call
\[
\qp = \frac{\bk\{x,y\}}{(yx - qxy)}
\]
the \textbf{quantum plane}.  We will apply Theorem \ref{thm:hombialg} to the quantum $n$-spaces with the following algebra morphisms.

Suppose $f \colon \{1,\ldots,n\} \to \bZ$ (or $f \colon \{1,2,\ldots\} \to \bZ$ if $n=\infty$) is an order-preserving function, i.e., $f(i) < f(j)$ whenever $i<j$.  For each $m \in Im(f) \cap \{1,\ldots,n\}$, pick a scalar $\lambda_m \in \bk$.  Given this data, define a map $\alpha \colon \qn \to \qn$ by setting
\begin{equation}
\label{alphax}
\alpha(x_i) =
\begin{cases}
\lambda_{f(i)}x_{f(i)} & \text{if $1 \leq f(i) \leq n$},\\
0 & \text{otherwise}.
\end{cases}
\end{equation}
for $1 \leq i \leq n$.  Since $f$ is order-preserving, the map $\alpha$ preserves the quantum commutation relation whenever $i < j$, so $\alpha$ is a well-defined algebra morphism.

By Theorem \ref{thm:hombialg} we have a Hom-associative algebra
\begin{equation}
\label{hqnspace}
\qnalpha = (\qn,\mualpha,\alpha),
\end{equation}
where $\mualpha = \alpha \circ \mu$ with $\mu$ the multiplication in $\qn$. We refer to the Hom-associative algebras $\qnalpha$ as the \textbf{Hom-quantum $n$-spaces} and $\qpalpha$ as the \textbf{Hom-quantum planes}. Since $\alpha$ depends on the scalars $\lambda_m$ and the function $f$, we have thus constructed a multi-parameter family $\{\qnalpha\}$ of Hom-quantum $n$-spaces.  We recover the quantum $n$-space $\qn$ by taking $f(i) = i$ for all $i$ and $\lambda_m = 1$ for all $m$.  We will revisit the Hom-quantum planes in Examples \ref{ex:hqg} and \ref{ex2:hqg} when we discuss their Hom-quantum geometry.
\qed
\end{example}

\begin{example}[\textbf{Fermionic Hom-quantum $n$-spaces}]
\label{ex:fhqspace}
In this example, we consider a fermionic version of the Hom-quantum $n$-spaces \eqref{hqnspace}.  We use the same notations as in Example \ref{ex:hqspace}.  The \textbf{fermionic quantum $n$-space} is defined as the unital associative algebra
\[
\qfn = \frac{\bk\{x_r\}_{r=1}^n}{(x_i^2 \text{ for all $i$},\, x_jx_i + qx_ix_j \text{ for $i<j$})}.
\]
The relation, $x_jx_i = -qx_ix_j$ for $i<j$, is called the \emph{fermionic quantum commutation relation}.  Unlike the quantum $n$-space $\qn$, the fermionic quantum $n$-space is finite-dimensional whenever $n<\infty$.

We can define an algebra morphism $\alpha \colon \qfn \to \qfn$ exactly as in \eqref{alphax}.  This map is well-defined because it clearly preserves the relations $x_i^2 = 0$ for all $i$ and the fermionic quantum commutation relation.  By Theorem \ref{thm:hombialg} we have a Hom-associative algebra
\[
\qfnalpha = (\qfn,\mualpha,\alpha),
\]
which we refer to as a \textbf{fermionic Hom-quantum $n$-space}.  The collection $\{\qfnalpha\}$ is a multi-parameter family of Hom-associative algebra deformations of the fermionic quantum $n$-space.  We recover the fermionic quantum $n$-space $\qfn$ by taking $f(i) = i$ for all $i$ and $\lambda_m = 1$ for all $m$.
\qed
\end{example}

\begin{example}[\textbf{Hom-quantum enveloping algebras of semisimple Lie algebras}]
\label{ex:uqg}
Fix an invertible scalar $q \in \bC$ with $q \not= \pm 1$.  Let $\fg$ be a complex semisimple Lie algebra of type $A_n$, $D_n$, $E_6$, $E_7$, or $E_8$.  Here we construct a multi-parameter, uncountable family of Hom-quantum groups $\uqg_\alpha$ from the quantum enveloping algebra $\uqg$ using Theorem \ref{thm:hombialg}.  The reader is referred to, e.g., \cite{hum,jac} for the classification of complex semisimple Lie algebras and discussions about Cartan matrices.  For example, the Lie algebra of type $A_n$ is $\slnplusone$, so in particular the case $A_1$ is $\sltwo$.

Let us first recall the bialgebra structure of the Drinfel'd-Jimbo quantum enveloping algebra $\uqg$ \cite{dri85,dri87,jimbo}.  The restriction on the type of $\fg$ implies that the Cartan matrix $(a_{ij})_{i,j=1}^n$ of $\fg$ is symmetric positive-definite with $a_{ii} = 2$ for all $i$ and $a_{ij} \in \{0,-1\}$ if $i\not= j$.  The quantum group $\uqg$ is generated as a unital associative algebra by $4n$ generators $\{E_i,F_i,K_i^{\pm 1}\}_{i=1}^n$ with relations:
\begin{equation}
\label{uqgrelations}
\begin{split}
K_iK_i^{-1} &= 1 = K_i^{-1}K_i,\quad [K_i,K_j] = 0,\\
K_iE_j &= q^{a_{ij}}E_jK_i,\quad K_iF_j = q^{-a_{ij}}F_jK_i,\quad
[E_i,F_j] = \delta_{ij}\frac{K_i - K_i^{-1}}{q - q^{-1}},\\
0 &= [E_i,E_j] = [F_i,F_j] \quad \text{if $a_{ij}=0$},\quad\text{and}\\
0 &= E_i^2E_j - [2]_qE_iE_jE_i + E_jE_i^2 =
F_i^2F_j - [2]_qF_iF_jF_i + F_jF_i^2
\end{split}
\end{equation}
if $a_{ij} = -1$.  Here the bracket is the commutator bracket, $[x,y] = xy - yx$, and $[2]_q = q + q^{-1}$.  The comultiplication on $\uqg$ is defined on the algebra generators by
\begin{equation}
\label{uqgcomult}
\begin{split}
\Delta(E_i) &= 1 \otimes E_i + E_i \otimes K_i, \\
\Delta(F_i) &= K_i^{-1} \otimes F_i + F_i \otimes 1,\\
\Delta(K_i) &= K_i \otimes K_i,\quad\text{and}\quad
\Delta(K_i^{-1}) = K_i^{-1} \otimes K_i^{-1}.
\end{split}
\end{equation}
For example, for $n \geq 1$ the Cartan matrix for $\slnplusone$, which is of type $A_n$, is $(a_{ij})_{i,j=1}^n$ with
\begin{equation}
\label{slcartan}
a_{ij} = \begin{cases}
2 & \text{if $i=j$},\\
-1 & \text{if $|i-j|=1$},\\
0 & \text{if $|i-j|>1$}.
\end{cases}
\end{equation}
In particular, when $n=1$ the Cartan matrix for $\sltwo$ has only one entry, namely, $a_{11}=2$.  In this case, the last two lines in \eqref{uqgrelations} are irrelevant.  So $\uq$ is generated as a unital associative algebra by $E$, $F$, $K$, and $K^{-1}$, satisfying the relations
\begin{equation}
\label{uqrelations}
\begin{split}
KK^{-1} &= 1 = K^{-1}K,\\
KE &= q^2EK,\quad KF = q^{-2}FK,\quad\text{and}\quad\\
EF - FE &= \frac{K - K^{-1}}{q - q^{-1}}.
\end{split}
\end{equation}

To use Theorem \ref{thm:hombialg} on the bialgebra $\uqg$, we need bialgebra morphisms on $\uqg$, which can be constructed as follows.  For each $i = 1, \ldots , n$, pick an invertible scalar $\lambda_i \in \bC$ and consider the map $\al \colon \uqg \to \uqg$ defined on the generators by
\begin{equation}
\label{alphalambdauqg}
\al(K_i^{\pm 1}) = K_i^{\pm 1},\quad
\al(E_i) = \lambda_i E_i,\quad\text{and}\quad
\al(F_i) = \lambda_i^{-1}F_i.
\end{equation}
It is immediate to check that $\al$ preserves all the relations in \eqref{uqgrelations} and \eqref{uqgcomult}, so $\al$ is a bialgebra automorphism on $\uqg$.  By Theorem \ref{thm:hombialg} there is a Hom-bialgebra
\begin{equation}
\label{uqgal}
\uqg_\al = (\uqg,\mu_\al,\Delta_\al,\al),
\end{equation}
where $\mu_\al = \al \circ \mu$ with $\mu$ the multiplication in $\uqg$ and $\Delta_\al = \Delta \circ \al$.  We refer to $\uqg_\al$ as a \textbf{Hom-quantum enveloping algebra} of $\fg$.   The collection $\{\uqg_\al \colon \lambda_1\cdots\lambda_n \not= 0\}$ is an $n$-parameter, uncountable family of Hom-bialgebra deformations of the quantum group $\uqg$.  We recover $\uqg$ by taking $\lambda_i = 1$ for all $i$.  When $\lambda_i \not= 1$ for at least one $i$ (i.e., $\al \not= Id$), the Hom-quantum enveloping algebra $\uqg_\al$ is simultaneously non-associative, non-coassociative, non-commutative, and non-cocommutative.

We will revisit the Hom-quantum enveloping algebras $\uqsln_\al$ in section \ref{sec:modules}, where we will construct Hom-type modules over them.  Moreover, in Example \ref{ex:hqg} we will discuss Hom-quantum geometry on the Hom-quantum planes $\qpalpha$ (Example \ref{ex:hqspace}) as they are act upon by the Hom-quantum enveloping algebras $\uq_\al$.
\qed
\end{example}

\begin{example}[\textbf{Hom-quantum Kac-Moody algebras}]
\label{ex:kacmoody}
In this example, we construct multi-parameter, uncountable classes of Hom-quantum enveloping algebras of symmetrizable Kac-Moody algebras.  The reader is referred to \cite{kac} for detailed discussions on Kac-Moody algebras and to \cite{dri85,dri87,jimbo} for the Drinfel'd-Jimbo quantum enveloping algebra $\uhg$.  Other expositions on $\uhg$ can be found in the books \cite{cp,es}.

Let $\fg$ be a symmetrizable Kac-Moody Lie algebra with Cartan matrix $A = (a_{ij})_{i,j=1}^n$.  Let us first recall the bialgebra structure on $\uhg$.  Let $D = \diag(d_1,\ldots,d_n)$ be the diagonal matrix whose entries are relatively prime non-negative integers such that $DA$ is symmetric.  Let $h$ be a formal variable and $\ch$ be the topological complex power series algebra over $h$ with the $h$-adic topology.  Set $q_i = e^{d_ih}$ for $1 \leq i \leq n$.  The quantum Kac-Moody algebra $\uhg$ is the topological unital associative algebra over $\ch$ generated by the $3n$ generators $\{H_i,X_i^{\pm}\}_{i=1}^n$ with relations:
\begin{equation}
\label{uhgrelations}
\begin{split}
[H_i,H_j] &= 0, \quad [H_i,X_j^{\pm}] = \pm a_{ij}X_j^{\pm},\quad
[X_i^+,X_j^-] = \delta_{ij}\frac{q_i^{H_i} - q_i^{-H_i}}{q_i - q_i^{-1}}, \quad\text{and}\\
0 &= \sum_{k=0}^{1-a_{ij}} (-1)^k\begin{bmatrix}1-a_{ij}\\ k\end{bmatrix}_{q_i} (X_i^{\pm})^k X_j^{\pm} (X_i^{\pm})^{1-a_{ij}-k} \quad \text{if $i \not= j$}.
\end{split}
\end{equation}
The $q_i$-binomial coefficient in \eqref{uhgrelations} is defined as follows:
\begin{equation}
\label{qinteger}
[n]_q = \frac{q^n - q^{-n}}{q - q^{-1}},\quad
[n]_q! = [1]_q[2]_q\cdots[n]_q,\quad\text{and}\quad
\ir = \frac{[n]_q!}{[r]_q![n-r]_q!}.
\end{equation}
The comultiplication on $\uhg$ is defined on the topological generators by
\begin{equation}
\label{uhgcomult}
\begin{split}
\Delta(H_i) &= H_i \otimes 1 + 1 \otimes H_i,\\
\Delta(X_i^+) &= 1 \otimes X_i^+ + X_i^+ \otimes q_i^{H_i},\quad\text{and}\\
\Delta(X_i^-) &= q_i^{-H_i} \otimes X_i^- + X_i^- \otimes 1.
\end{split}
\end{equation}

To use Theorem \ref{thm:hombialg} on the quantum Kac-Moody algebra $\uhg$, we proceed essentially as in the previous example.  For each $i = 1, \ldots , n$, pick an invertible element $p_i \in \ch$, i.e., $p_i$ is a complex power series in $h$ with non-zero constant term.  Consider the $\ch$-linear map $\alpha_p \colon \uhg \to \uhg$ defined on the generators by
\[
\alpha_p(H_i) = H_i, \quad \alpha_p(X_i^+) = p_i X_i^+, \quad\text{and}\quad \alpha_p(X_i^-) = p_i^{-1} X_i^-.
\]
It is easy to see that $\alpha_p$ preserves all the relations in \eqref{uhgrelations} and \eqref{uhgcomult}, so $\alpha_p$ is a bialgebra automorphism on $\uhg$.  By Theorem \ref{thm:hombialg}, there is a Hom-bialgebra
\[
\uhg_{\alpha_p} = (\uhg,\mu_\alpha,\Delta_\alpha,\alpha_p),
\]
where $\mu$ is the multiplication on $\uhg$, $\mu_\alpha = \alpha_p \circ \mu$, and $\Delta_\alpha = \Delta \circ \alpha_p$.  The collection $\{\uhg_{\alpha_p}\}$ is an $n$-parameter, uncountable family of Hom-bialgebra deformations of the quantum Kac-Moody algebra $\uhg$.  We recover $\uhg$ from $\uhg_{\alpha_p}$ when all $n$ parameters $\{p_i\}_{i=1}^n$ are equal to $1$.  When at least one $p_i \not= 1$, the Hom-bialgebra $\uhg_{\alpha_p}$ is simultaneously non-associative, non-coassociative, non-commutative, and non-cocommutative.
\qed
\end{example}

The following result is the second Twisting Principle \ref{twistingprinciple} for Hom-(co)associative (co)algebras and Hom-bialgebras.  It says that each Hom-(co)associative (co)algebra (or Hom-bialgebra) gives rise to a derived sequence of Hom-(co)associative (co)algebras (or Hom-bialgebras) with twisted (co)multiplications and twisting maps.

\begin{theorem}
\label{thm:secondtp}
Let $(A,\mu,\alpha)$ be a Hom-associative algebra, $(C,\Delta,\alpha)$ be a Hom-coassociative coalgebra, and $(H,\mu,\Delta,\alpha)$ be a Hom-bialgebra.  Then for each integer $n \geq 0$:
\begin{enumerate}
\item
$A^n = (A,\mu^{(n)} = \alpha^{2^n-1}\circ \mu,\alpha^{2^n})$ is a Hom-associative algebra.
\item
$C^n = (C,\Delta^{(n)} = \Delta \circ \alpha^{2^n-1},\alpha^{2^n})$ is a Hom-coassociative coalgebra.
\item
$H^n = (H,\mu^{(n)},\Delta^{(n)},\alpha^{2^n})$ is a Hom-bialgebra, where $\mu^{(n)} = \alpha^{2^n-1}\circ \mu$ and $\Delta^{(n)} = \Delta \circ \alpha^{2^n-1}$.
\end{enumerate}
\end{theorem}

\begin{proof}
Consider the first assertion.  Note that $A^0 = A$, $A^1 = (A,\mu^{(1)}=\alpha\circ\mu,\alpha^2)$, and $A^{n+1} = (A^n)^1$ because
\[
\mu^{(n+1)} = \alpha^{2^n} \circ \mu^{(n)} \quad\text{and}\quad
\alpha^{2^{n+1}} = (\alpha^{2^n})^2.
\]
Therefore, by an induction argument, it suffices to prove the case $n = 1$.  In other words, we need to check the two axioms \eqref{homassaxioms} with $\mu^{(1)}$ and $\alpha^2$.  First, the multiplicativity for $A^1 = (A,\mu^{(1)}=\alpha\circ\mu,\alpha^2)$ says
\[
\alpha^2 \circ \mu^{(1)} = \mu^{(1)} \circ (\alpha^2)^{\otimes 2}.
\]
This is true by the following commutative diagram:
\[
\nicearrow
\xymatrix{
A \otimes A \ar[rr]^-{\mu} \ar[d]_-{\alpha^{\otimes 2}} & & A \ar[rr]^-{\alpha} \ar[d]^-{\alpha} & & A \ar[d]^-{\alpha}\\
A \otimes A \ar[rr]^-{\mu} \ar[d]_-{\alpha^{\otimes 2}} & & A \ar[d]^-{\alpha} & & A \ar[d]^-{\alpha}\\
A \otimes A \ar[rr]^-{\mu} & & A \ar[rr]^-{\alpha} & & A.
}
\]
The two rectangles on the left are commutative by the multiplicativity in $A$, and the right rectangle is trivially commutative.

Next, the Hom-associativity for $A^1$ says
\begin{equation}
\label{homassA1}
\mu^{(1)} \circ (\alpha^2 \otimes \mu^{(1)}) = \mu^{(1)} \circ (\mu^{(1)} \otimes \alpha^2).
\end{equation}
This is true by the following commutative diagram:
\begin{equation}
\label{a1homass}
\nicearrow
\xymatrix{
A \otimes A \otimes A \ar[rr]^-{\mu \otimes \alpha} \ar[d]_-{\alpha \otimes \mu} & & A \otimes A \ar[rr]^-{\alpha^{\otimes 2}} \ar[d]^-{\mu} & & A \otimes A \ar[d]^-{\mu}\\
A \otimes A \ar[rr]^-{\mu} \ar[d]_-{\alpha^{\otimes 2}} & & A \ar[rr]^-{\alpha} \ar[d]^-{\alpha} & & A \ar[d]^-{\alpha}\\
A \otimes A \ar[rr]^-{\mu} & & A \ar[rr]^-{\alpha} & & A.
}
\end{equation}
The composition along the left and then the bottom edges is the left-hand side in \eqref{homassA1}.  The composition along the top and then the right edges is the right-hand side in \eqref{homassA1}.  The upper left rectangle is commutative by the Hom-associativity in $A$.  The lower left and the upper right rectangles are commutative by the multiplicativity in $A$.  The lower right rectangle is trivially commutative.  We have shown that $A^1$ is a Hom-associative algebra, as desired.

The second assertion, that each $C^n$ is a Hom-coassociative coalgebra, is proved by the exact dual argument.  Indeed, for the case $C^1$, simply invert the arrows and replaces $\mu$ by $\Delta$ in the two commutative diagrams above.  Then an induction argument proves the assertion because, as above, $C^0 = C$ and $C^{n+1} = (C^n)^1$.

For the last assertion, since $H^0 = H$ and $H^{(n+1)} = (H^n)^1$ as before, by an induction argument it suffices to prove the case $n=1$.  From the first two assertions, we already know that $H^1$ is both a Hom-associative algebra and a Hom-coassociative coalgebra.  Thus, it remains to show the compatibility condition \eqref{def:hombi} for $\mu^{(1)}$ and $\Delta^{(1)}$, which says
\begin{equation}
\label{mudelta1}
\Delta^{(1)} \circ \mu^{(1)} = (\mu^{(1)})^{\otimes 2} \circ (2~3) \circ (\Delta^{(1)})^{\otimes 2},
\end{equation}
where $(2~3)$ is the abbreviation for $Id \otimes \tau \otimes Id$ (permuting the middle two entries).  The condition \eqref{mudelta1} holds by the following commutative diagram:
\begin{equation}
\label{compatmu1}
\nicearrow
\xymatrix{
H^{\otimes 2} \ar[rr]^-{\alpha^{\otimes 2}} \ar[d]_-{\mu} & & H^{\otimes 2} \ar[rr]^-{\Delta^{\otimes 2}} \ar[d]^-{\alpha^{\otimes 2}} & & H^{\otimes 4} \ar[rr]^-{(2~3)} \ar[d]^-{\alpha^{\otimes 4}} & & H^{\otimes 4} \ar[d]^-{\alpha^{\otimes 4}}\\
H \ar[d]_-{\alpha} & & H^{\otimes 2} \ar[rr]^-{\Delta^{\otimes 2}} \ar[d]^-{\mu} & & H^{\otimes 4} \ar[rr]^-{(2~3)} & & H^{\otimes 4} \ar[d]^-{\mu^{\otimes 2}}\\
H \ar[rr]^-{\alpha} & & H \ar[rrrr]^-{\Delta} & & & & H^{\otimes 2}.
}
\end{equation}
Along the left and the bottom edges of the big rectangle, the composition is the left-hand side of \eqref{mudelta1}.  Along the top and the right edges of the big rectangle, the composition is the right-hand side of \eqref{mudelta1}. The left square is commutative by the multiplicativity in $H$ twice, and the upper middle rectangle is commutative by the comultiplicativity in $H$ twice.  The bottom right rectangle is commutative by the compatibility between $\mu$ and $\Delta$ in $H$, i.e., \eqref{def:hombi}.  The upper right rectangle is trivially commutative.
\end{proof}

In the context of Theorem \ref{thm:secondtp}, we call $A^n = (A,\mu^{(n)} = \alpha^{2^n-1}\circ \mu,\alpha^{2^n})$ the $n$th \textbf{derived Hom-associative algebra} of $A$, $C^n = (C,\Delta^{(n)} = \Delta \circ \alpha^{2^n-1},\alpha^{2^n})$ the $n$th \textbf{derived Hom-coassociative coalgebra} of $C$, and $H^n = (H,\mu^{(n)},\Delta^{(n)},\alpha^{2^n})$ the $n$th \textbf{derived Hom-bialgebra} of $H$.  Theorem \ref{thm:secondtp} will be used several times in the next few sections.

Some remarks are in order.

\begin{remark}
\label{rk:derivedhomalg}
\begin{enumerate}
\item
In the proofs of the Hom-associativity in $A^1$ \eqref{a1homass} and of the compatibility condition in $H^1$ \eqref{compatmu1}, the (co)multiplicativity of $\alpha$ with respect to $\mu$ and $\Delta$ is crucial.  This is one reason why we insist on the (co)multiplicativity condition in the definition of a Hom-(co)associative (co)algebra, even though it was not part of the original definition in \cite{ms,ms4}
\item
Theorem \ref{thm:secondtp} has obvious analogs for $G$-Hom-(co)associative (co)algebras \cite{ms2,ms4} with essentially the same formulations, provided the twisting map $\alpha$ is assumed to be (co)multiplicative.  Here $G$ is any subgroup of the symmetric group on $3$ letters.  These objects are the Hom-type analogs of the $G$-(co)associative (co)algebras in \cite{gr}.  In particular, the second Twisting Principle \ref{twistingprinciple} applies to Hom-Lie and Hom-Lie-admissible Hom-(co)algebras.  It is also not hard to check that there are analogs of Theorem \ref{thm:secondtp} for Hom-Novikov algebras \cite{hb,yau5}, Hom-alternative algebras \cite{mak}, and Hom-Lie bialgebras \cite{yau8}.
\item
If we first use Theorem \ref{thm:hombialg} (the first Twisting Principle) and then Theorem \ref{thm:secondtp} (the second Twisting Principle), we actually obtain a special case of Theorem \ref{thm:hombialg}.  Indeed, suppose $(A,\mu)$ is an associative algebra and $\alpha \colon A \to A$ is an algebra morphism.  Applying Theorem \ref{thm:secondtp} to the Hom-associative algebra $A_\alpha = (A,\mu_\alpha = \alpha\circ\mu,\alpha)$ in Theorem \ref{thm:hombialg}, we obtain
\[
(A_\alpha)^1 = (A,\mu_\alpha^{(1)} = \alpha \circ \mu_\alpha = \alpha^2 \circ \mu,\alpha^2).
\]
In other words, we have $(A_\alpha)^1 = A_{\alpha^2}$, which is Theorem \ref{thm:hombialg} for $\alpha^2$.  Inductively, we have
\[
(A_\alpha)^n = A_{\alpha^{2^n}}
\]
for all $n$.  Similar remarks apply to coassociative coalgebras and bialgebras.
\item
On the other hand, there are examples of Hom-associative algebras and Hom-bialgebras that are not of the form $A_\alpha$, in which case Theorem \ref{thm:secondtp} produces Hom-associative algebras and Hom-bialgebras that cannot be obtained from Theorem \ref{thm:hombialg}. For example, the free Hom-associative algebra generated by a non-zero module \cite{yau3} and the universal enveloping Hom-bialgebra of a Hom-Lie algebra \cite{yau,yau3} are not of the form $A_\alpha$.  In fact, in $A_\alpha = (A,\mu_\alpha,\alpha)$ (as in Theorem \ref{thm:hombialg}), the image of the multiplication $\mu_\alpha = \alpha\circ\mu$ is contained in that of $\alpha$.  So given a Hom-associative algebra (or a Hom-bialgebra), as long as the image of its multiplication is not contained in that of its  twisting map $\alpha$, the Hom-associative algebra (or the Hom-bialgebra) in question is not of the form $A_\alpha$.
\end{enumerate}
\end{remark}

\section{Quasi-triangular and cobraided Hom-bialgebras}
\label{sec:qthb}

The purpose of this section is to establish the second Twisting Principle \ref{twistingprinciple} for quasi-triangular and cobraided Hom-bialgebras.  They are Hom-type analogs of Drinfel'd's quasi-triangular bialgebras \cite{dri87} and of the dual objects of cobraided bialgebras \cite{hay,lt,majid91,sch}.  These Hom-quantum groups were first studied in \cite{yau9,yau10}, where the first Twisting Principle was established for them and where many examples can be found.

To define quasi-triangular Hom-bialgebras, we need the following notion of units from \cite{fg2}.

\begin{definition}
\label{def:weakunit}
\begin{enumerate}
\item
Let $(A,\mu,\alpha)$ be a Hom-associative algebra.  A \textbf{weak unit} \cite{fg2} of $A$ is an element $c \in A$ such that $\alpha(x) = cx = xc$ for all $x \in A$.  In this case, we call $(A,\mu,\alpha,c)$ a \textbf{weakly unital Hom-associative algebra}.
\item
Let $(A,\mu,\alpha,c)$ be a weakly unital Hom-associative algebra and $R \in A^{\otimes 2}$.  Define the following elements in $A^{\otimes 3}$:
\begin{equation}
\label{R123}
R_{12} = R \otimes c,\quad R_{23} = c \otimes R, \quad\text{and}\quad R_{13} = (\tau \otimes Id)(R_{23}).
\end{equation}
\end{enumerate}
\end{definition}

\begin{example}[\cite{fg2} Example 2.2]\label{ex:weakunit}
If $(A,\mu,1)$ is a unital associative algebra, then the Hom-associative algebra $A_\alpha = (A,\mu_\alpha,\alpha)$ (Theorem \ref{thm:hombialg}) has a weak unit $c = 1$.\qed
\end{example}

\begin{definition}
\label{def:qthb}
A \textbf{quasi-triangular Hom-bialgebra} \cite{yau9} is a tuple $(H,\mu,\Delta,\alpha,c,R)$ in which $(H,\mu,\Delta,\alpha)$ is a Hom-bialgebra, $c$ is a weak unit of $(H,\mu,\alpha)$, and $R \in H^{\otimes 2}$ satisfies the following three axioms:
\begin{subequations}
\label{qtaxioms}
\begin{align}
(\Delta \otimes \alpha)(R) &= R_{13}R_{23},\label{R1323}\\
(\alpha \otimes \Delta)(R) &= R_{13}R_{12},\quad\text{and}\label{R1312}\\
[\Delta^{op}(x)]R &= R\Delta(x)\label{RDelta}
\end{align}
\end{subequations}
for all $x \in H$, where $\Delta^{op} = \tau\circ\Delta$ and the elements $R_{12}$, $R_{13}$, and $R_{23}$ are defined in \eqref{R123}.  We call $R$ the \textbf{Hom-braiding element} of $H$.
\end{definition}

In the axioms \eqref{qtaxioms}, the multiplication is done in each tensor factor.  Thus, with $R = \sum s_i \otimes t_i \in H^{\otimes 2}$, the three axioms in \eqref{qtaxioms} can be restated respectively as:
\begin{align*}
\sum s_i' \otimes s_i'' \otimes \alpha(t_i) &= \sum \alpha(s_i) \otimes \alpha(s_j) \otimes t_it_j,\\
\sum \alpha(s_i) \otimes t_i' \otimes t_i'' &= \sum s_js_i \otimes \alpha(t_i) \otimes \alpha(t_j),\quad\text{and}\\
\sum x_2s_i \otimes x_1t_i &= \sum s_ix_1 \otimes t_ix_2,
\end{align*}
where $\Delta(x) = \sum x_1 \otimes x_2$, $\Delta(s_i) = \sum s_i' \otimes s_i''$, and $\Delta(t_i) = \sum t_i' \otimes t_i''$.

The following result is the second Twisting Principle \ref{twistingprinciple} for quasi-triangular Hom-bialgebras.  It says that each quasi-triangular Hom-bialgebra gives rise to a derived sequence of quasi-triangular Hom-bialgebras with twisted (co)multiplications and twisting maps.  If, in addition, the twisting map $\alpha$ is surjective, then there is a derived double-sequence of quasi-triangular Hom-bialgebras, where the Hom-braiding elements are also twisted.

\begin{theorem}
\label{thm:qtsecond}
Let $(H,\mu,\Delta,\alpha,c,R)$ be a quasi-triangular Hom-bialgebra.
\begin{enumerate}
\item
Then $H^n = (H,\mu^{(n)},\Delta^{(n)},\alpha^{2^n},c,R)$ is a quasi-triangular Hom-bialgebra for each integer $n \geq 0$, where $\mu^{(n)} = \alpha^{2^n-1}\circ\mu$ and $\Delta^{(n)} = \Delta \circ \alpha^{2^n-1}$.
\item
If $\alpha$ is surjective, then $H^{n,k} = (H,\mu^{(n)},\Delta^{(n)},\alpha^{2^n},c,\rk)$ is a quasi-triangular Hom-bialgebra for each pair of integers $n,k \geq 0$, where $\rk = (\alpha^{k} \otimes \alpha^{k})(R)$.
\end{enumerate}
\end{theorem}

\begin{proof}
Consider the first assertion.  By Theorem \ref{thm:secondtp} we already know that $H^n$ is a Hom-bialgebra.  As in the proof of Theorem \ref{thm:secondtp}, since $H^0 = H$ and $H^{n+1} = (H^n)^1$, by an induction argument it suffices to prove the case $n=1$.  In other words, we need to show that $c$ is a weak unit of $H^1$ and that the axioms \eqref{qtaxioms} are satisfied in $H^1$.

To see that $c$ is a weak unit of $H^1$, pick any element $x \in H$.  With $\mu$ written as concatenation, we compute as follows:
\[
\begin{split}
\mu^{(1)}(x,c) &= \alpha(xc)\\
&= \alpha^2(x) \\
&= \alpha(cx) \\
&= \mu^{(1)}(c,x).
\end{split}
\]
This proves that $c$ is still a weak unit of $H^1$.  It remains to prove the three axioms \eqref{qtaxioms} for $H^1$.

The axiom \eqref{R1323} for $H^1$ says
\[
(\Delta^{(1)} \otimes \alpha^2)(R) = \mu^{(1)}(R_{13},R_{23}),
\]
where the right-hand side means applying $\mu^{(1)} = \alpha\circ\mu$ in each of the three tensor factors.  With $R = \sum s_i \otimes t_i$, we compute as follows:
\[
\begin{split}
(\Delta^{(1)} \otimes \alpha^2)(R)
&= ((\alpha^{\otimes 2} \circ \Delta) \otimes \alpha^2)(R)\\
&= \alpha^{\otimes 3}((\Delta \otimes \alpha)(R))\\
&= \alpha^{\otimes 3}\left(R_{13}R_{23}\right)\\
&= \sum \alpha(s_ic) \otimes \alpha(cs_j) \otimes \alpha(t_it_j)\\
&= \sum \mu^{(1)}(s_i,c) \otimes \mu^{(1)}(c,s_j) \otimes \mu^{(1)}(t_i,t_j)\\
&= \mu^{(1)}(R_{13},R_{23}).
\end{split}
\]
In the third equality above, we used the axiom \eqref{R1323} for $H$.  Likewise, the axiom \eqref{R1312} for $H^1$ says
\[
(\alpha^2 \otimes \Delta^{(1)})(R) = \mu^{(1)}(R_{13},R_{12}).
\]
This is true by the following computation:
\[
\begin{split}
(\alpha^2 \otimes \Delta^{(1)})(R)
&= \alpha^{\otimes 3}((\alpha \otimes \Delta)(R))\\
&= \alpha^{\otimes 3}(R_{13}R_{12})\\
&= \sum \alpha(s_js_i) \otimes \alpha(ct_i) \otimes \alpha(t_jc)\\
&= \mu^{(1)}(R_{13},R_{12}).
\end{split}
\]
In the second equality above, we used the axiom \eqref{R1312} for $H$.  Finally, the axiom \eqref{RDelta} for $H^1$ says
\[
\mu^{(1)}((\Delta^{(1)})^{op}(x),R) = \mu^{(1)}(R,\Delta^{(1)}(x)).
\]
This is true by the following computation:
\[
\begin{split}
\mu^{(1)}((\Delta^{(1)})^{op}(x),R)
&= \alpha^{\otimes 2}\left(\Delta^{op}(\alpha(x))R\right)\\
&= \alpha^{\otimes 2}(R\Delta(\alpha(x)))\\
&= \mu^{(1)}(R,\Delta^{(1)}(x)).
\end{split}
\]
In the second equality above, we used the axiom \eqref{RDelta} for $H$, applied to the element $\alpha(x)$.  We have shown that $H^1$ is a quasi-triangular Hom-bialgebra.  By the remark in the first paragraph of this proof, we have proved the first assertion of the Theorem.

Next consider the second assertion of the Theorem.  It is proved in \cite{yau9} (Theorem 3.3) that for a quasi-triangular Hom-bialgebra $H$ with $\alpha$ surjective, the object $H^{0,k} = (H,\mu,\Delta,\alpha,c,\rk)$ is also a quasi-triangular Hom-bialgebra for each $k \geq 0$.  Now apply the first assertion to $H^{0,k}$, and observe that $(H^{0,k})^n = H^{n,k}$.
\end{proof}

Now we consider the dual objects of cobraided Hom-bialgebras.

\begin{definition}
\label{def:dqt}
A \textbf{cobraided Hom-bialgebra} \cite{yau10} is a quintuple $(H,\mu,\Delta,\alpha,R)$ in which $(H,\mu,\Delta,\alpha)$ is a Hom-bialgebra and $R$ is a bilinear form on $H$ (i.e., $R \in \Hom(H^{\otimes 2},\bk)$), satisfying the following three axioms:
\begin{subequations}
\label{dqtaxioms}
\begin{align}
R \circ (\mu \otimes \alpha) &= R^{\otimes 2} \circ (2~3) \circ (\alpha^{\otimes 2} \otimes \Delta),\label{axiom1}\\
R \circ (\alpha \otimes \mu) &= R^{\otimes 2} \circ (2~3~4) \circ (\Delta \otimes \alpha^{\otimes 2}) ,\quad\text{and}\label{axiom2}\\
(\mu \otimes R) \circ (1~2~3) \circ \Delta^{\otimes 2} &= (R \otimes \mu) \circ (2~3) \circ \Delta^{\otimes 2}.\label{axiom3}
\end{align}
\end{subequations}
Here $(2~3)$, $(2~3~4)$, and $(1~2~3)$ mean permutations of the tensor factors.  We call $R$ the \textbf{Hom-cobraiding form} of $H$.
\end{definition}

In terms of elements $x,y,z \in H$, the three axioms \eqref{dqtaxioms} can be restated respectively as:
\begin{align*}
\label{dqtax}
R(xy \otimes \alpha(z)) &= \sum_{(z)} R(\alpha(x) \otimes z_1)R(\alpha(y) \otimes z_2),\\
R(\alpha(x) \otimes yz) &= \sum_{(x)} R(x_1 \otimes \alpha(z))R(x_2 \otimes \alpha(y)),\quad\text{and}\\
\sum_{(x)(y)} y_1x_1R(x_2 \otimes y_2) &= \sum_{(x)(y)} R(x_1 \otimes y_1)x_2y_2.
\end{align*}

The following result is the second Twisting Principle \ref{twistingprinciple} for cobraided Hom-bialgebras.  It says that each cobraided Hom-bialgebra gives rise to a derived sequence of cobraided Hom-bialgebras with twisted (co)multiplications and twisting maps.  If, in addition, the twisting map $\alpha$ is injective, then there is a derived double-sequence of cobraided Hom-bialgebras, where the Hom-cobraiding forms are also twisted.

\begin{theorem}
\label{thm:cobraidedsecond}
Let $(H,\mu,\Delta,\alpha,R)$ be a cobraided Hom-bialgebra.
\begin{enumerate}
\item
Then $H^n = (H,\mu^{(n)},\Delta^{(n)},\alpha^{2^n},R)$ is a cobraided Hom-bialgebra for each integer $n \geq 0$, where $\mu^{(n)} = \alpha^{2^n-1}\circ\mu$ and $\Delta^{(n)} = \Delta \circ \alpha^{2^n-1}$.
\item
If $\alpha$ is injective, then $H^{n,k} = (H,\mu^{(n)},\Delta^{(n)},\alpha^{2^n},\rk)$ is a cobraided Hom-bialgebra for each pair of integers $n,k \geq 0$, where $\rk = R \circ (\alpha^k \otimes \alpha^k)$.
\end{enumerate}
\end{theorem}

\begin{proof}
Consider the first assertion.  As in the proof of Theorem \ref{thm:qtsecond}, by an induction argument it suffices to prove that $H^1$ is a cobraided Hom-bialgebra.  Since $H^1$ is known to be a Hom-bialgebra by Theorem \ref{thm:secondtp}, it remains to establish the three conditions \eqref{dqtaxioms} for $H^1$.

The axiom \eqref{axiom1} for $H^1$ says
\[
R \circ (\mu^{(1)} \otimes \alpha^2) = R^{\otimes 2} \circ (2~3) \circ \left((\alpha^2)^{\otimes 2} \otimes \Delta^{(1)}\right).
\]
This is true by the following commutative diagram:
\[
\nicearrow
\xymatrix{
H^{\otimes 3} \ar[rr]^-{\alpha^{\otimes 3}} \ar[dd]_-{\mu^{(1)} \otimes \alpha^2} & & H^{\otimes 3} \ar[rr]^-{\alpha^{\otimes 2} \otimes \Delta} \ar[dd]_-{\mu \otimes \alpha} & & H^{\otimes 4} \ar[d]^-{(2~3)}\\
& & & & H^{\otimes 4} \ar[d]^-{R^{\otimes 2}}\\
H^{\otimes 2} \ar@{=}[rr] & & H^{\otimes 2} \ar[rr]^-{R} & & \bk.
}
\]
The left square is commutative by definition.  The right square is commutative by the axiom \eqref{axiom1} for $H$.  Likewise, the axiom \eqref{axiom2} for $H^1$ says
\[
R \circ (\alpha^2 \otimes \mu^{(1)}) = R^{\otimes 2} \circ (2~3~4) \circ \left(\Delta^{(1)} \otimes (\alpha^2)^{\otimes 2}\right).
\]
This is true by the following commutative diagram:
\[
\nicearrow
\xymatrix{
H^{\otimes 3} \ar[rr]^-{\alpha^{\otimes 3}} \ar[dd]_-{\alpha^2 \otimes \mu^{(1)}} & & H^{\otimes 3} \ar[rr]^-{\Delta \otimes \alpha^{\otimes 2}} \ar[dd]_-{\alpha \otimes \mu} & & H^{\otimes 4} \ar[d]^-{(2~3~4)}\\
& & & & H^{\otimes 4} \ar[d]^-{R^{\otimes 2}}\\
H^{\otimes 2} \ar@{=}[rr] & & H^{\otimes 2} \ar[rr]^-{R} & & \bk.
}
\]
Again the left square is commutative by definition, and the right square is commutative by the axiom \eqref{axiom2} for $H$.  Finally, the axiom \eqref{axiom3} for $H^1$ says
\[
(\mu^{(1)} \otimes R) \circ (1~2~3) \circ (\Delta^{(1)})^{\otimes 2} = (R \otimes \mu^{(1)}) \circ (2~3) \circ (\Delta^{(1)})^{\otimes 2}.
\]
This is true by the following commutative diagram:
\[
\nicearrow
\xymatrix{
H^{\otimes 2} \ar[rr]^-{\alpha^{\otimes 2}} \ar[d]_-{(\Delta^{(1)})^{\otimes 2}} & & H^{\otimes 2} \ar[rr]^-{\Delta^{\otimes 2}} \ar[d]_-{\Delta^{\otimes 2}} & & H^{\otimes 4} \ar[d]^-{(2~3)}\\
H^{\otimes 4} \ar[d]_-{(1~2~3)} & & H^{\otimes 4} \ar[d]_-{(1~2~3)} & & H^{\otimes 4} \ar[d]^-{R \otimes \mu}\\
H^{\otimes 4} \ar[d]_-{\mu^{(1)} \otimes R} & & H^{\otimes 4} \ar[rr]^-{\mu \otimes R} \ar[d]_-{\mu^{(1)} \otimes R} & & H \ar[d]^-{\alpha}\\
H \ar@{=}[rr] & & H \ar@{=}[rr] & & H.
}
\]
The left rectangle and the bottom right rectangle are commutative by definition.  The top right square is commutative by the axiom \eqref{axiom3} for $H$.  This finishes the proof of the first assertion of the Theorem.

Next consider the second assertion.  It is proved in \cite{yau10} (Theorem 5.1) that for a cobraided Hom-bialgebra $H$ with $\alpha$ injective, the object $H^{0,k} = (H,\mu,\Delta,\alpha,\rk)$ is also a cobraided Hom-bialgebra for each $k \geq 0$.  Now apply the first assertion to $H^{0,k}$, and observe that $(H^{0,k})^n = H^{n,k}$.
\end{proof}

\section{Representations of Hom-algebras}
\label{sec:modules}

The purpose of this section is to study (co)modules over Hom-(co)associative (co)algebras.  We first establish the second Twisting Principle \ref{twistingprinciple} for modules over Hom-associative algebras.  In particular, we show that each module over a Hom-associative algebra gives rise to a derived double-sequence of modules with twisted actions (Corollary \ref{cor:twisthommodule}).  For the first Twisting Principle \ref{twistingprinciple}, starting with a module in the usual sense and a suitable pair of morphisms, we obtain a derived double-sequence of modules over Hom-associative algebras (Theorem \ref{thm:twistmodule}).
As examples, we construct multi-parameter, uncountable classes of modules over the Hom-quantum enveloping algebras $\uqsln_\al$ (Examples \ref{ex:uqmodule} - \ref{ex:slnmodule}).  The Twisting Principles for modules over Hom-associative algebras can be dualized to comodules over Hom-coassociative coalgebras (Theorems \ref{thm:twisthomco} and \ref{thm:twistco}).

\begin{definition}
\label{def:hommodule}
Let $(A,\mu_A,\alpha_A)$ be a Hom-associative algebra and $(C,\Delta_C,\alpha_C)$ be a Hom-coassociative coalgebra (Definition \ref{def:homas}).
\begin{enumerate}
\item
A \textbf{Hom-module} is a pair $(M,\alpha_M)$ consisting of a $\bk$-module $M$ and a linear self-map $\alpha_M \colon M \to M$.  A morphism of Hom-modules $f \colon (M,\alpha_M) \to (N,\alpha_N)$ is a morphism of the underlying $\bk$-modules that is compatible with the twisting maps, in the sense that $\alpha_N \circ f = f \circ \alpha_M$.  The tensor product of two Hom-modules $M$ and $N$ is the pair $(M \otimes N, \alpha_M \otimes \alpha_N)$.
\item
An \textbf{$A$-module} is a Hom-module $(M,\alpha_M)$ together with a linear map $\rho \colon A \otimes M \to M$, called the \textbf{structure map}, such that
\begin{equation}
\label{eq:moduleaxiom}
\begin{split}
\alpha_M \circ \rho &= \rho \circ (\alpha_A \otimes \alpha_M) \quad \text{(multiplicativity)},\\
\rho \circ (\alpha_A \otimes \rho) &= \rho \circ (\mu_A \otimes \alpha_M) \quad \text{(Hom-associativity)}.
\end{split}
\end{equation}
A morphism $f \colon M \to N$ of $A$-modules is a morphism of the underlying Hom-modules that is compatible with the structure maps, in the sense that
\begin{equation}
\label{eq:modmorphism}
f \circ \rho_M = \rho_N \circ (Id_A \otimes f).
\end{equation}
\item
A \textbf{$C$-comodule} is a Hom-module $(M,\alpha_M)$ together with a linear map $\rho \colon M \to C \otimes M$, called the \textbf{structure map}, such that
\begin{equation}
\label{eq:comoduleaxiom}
\begin{split}
\rho \circ \alpha_M &= (\alpha_C \otimes \alpha_M) \circ \rho \quad \text{(comultiplicativity)},\\
(\alpha_C \otimes \rho) \circ \rho &= (\Delta_C \otimes \alpha_M) \circ \rho \quad \text{(Hom-coassociativity)}.
\end{split}
\end{equation}
A morphism $f \colon M \to N$ of $C$-comodules is a morphism of the underlying Hom-modules that is compatible with the structure maps, in the sense that $(Id_C \otimes f) \circ \rho_M = \rho_N \circ f$.
\end{enumerate}
\end{definition}

Note that the (co)multiplicativity conditions in \eqref{eq:moduleaxiom} and \eqref{eq:comoduleaxiom} are equivalent to $\rho$ being a morphism of Hom-modules.  The following result is a version of the second Twisting Principle \ref{twistingprinciple} for modules over a Hom-associative algebra.

\begin{theorem}
\label{thm:twisthommodule}
Let $(A,\mu,\alpha_A)$ be a Hom-associative algebra and $(M,\alpha_M)$ be an $A$-module with structure map $\rho \colon A \otimes M \to M$.  For each integer $n \geq 0$, define the map
\begin{equation}
\label{rhon0}
\rho^{n,0} = \rho \circ (\alpha_A^n \otimes Id_M) \colon A \otimes M \to M.
\end{equation}
Then each $\rho^{n,0}$ gives the Hom-module $M$ the structure of an $A$-module.
\end{theorem}

Note that if $\rho(a,m)$ is written as $am$, then $\rho^{n,0}(a,m) = \alpha_A^n(a)m$.  The meaning of the $0$ in $\rho^{n,0}$ will be made clear in Corollary \ref{cor:twisthommodule} below.

\begin{proof}
Since $\rho^{0,0} = \rho$, $\rho^{1,0} = \rho \circ (\alpha_A \otimes Id_M)$, and
\[
\rho^{n+1,0} = \rho^{n,0} \circ (\alpha_A \otimes Id_M) = (\rho^{n,0})^{1,0},
\]
by an induction argument it suffices to prove the Theorem for the case $n = 1$.  In other words, we need to prove that $\rho^{1,0} = \rho \circ (\alpha_A \otimes Id_M)$ gives $M$ the structure of an $A$-module.  To check this, first observe that the multiplicativity \eqref{eq:moduleaxiom} of $\rho^{1,0}$ says
\[
\alpha_M \circ \rho^{1,0} = \rho^{1,0} \circ (\alpha_A \otimes \alpha_M).
\]
This is true by the following commutative diagram:
\[
\nicearrow
\xymatrix{
A \otimes M \ar[rr]^-{\alpha_A \otimes \alpha_M} \ar[d]_-{\alpha_A \otimes Id_M} & & A \otimes M \ar[d]^-{\alpha_A \otimes Id_M}\\
A \otimes M \ar[d]_-{\rho} \ar[rr]^-{\alpha_A \otimes \alpha_M} & & A \otimes M \ar[d]^-{\rho}\\
M \ar[rr]_-{\alpha_M} & & M.
}
\]
The top rectangle is commutative because both compositions are equal to $\alpha_A^2 \otimes \alpha_M$.  The bottom rectangle is commutative by the multiplicativity of $\rho$.

Likewise, the Hom-associativity \eqref{eq:moduleaxiom} of $\rho^{1,0}$ says
\[
\rho^{1,0} \circ (\alpha_A \otimes \rho^{1,0}) = \rho^{1,0} \circ (\mu_A \otimes \alpha_M).
\]
This is true by the following commutative diagram:
\[
\nicearrow
\xymatrix{
A \otimes A \otimes M \ar[rrrr]^-{\mu_A \otimes \alpha_M} \ar[d]_-{\alpha_A \otimes \alpha_A \otimes Id_M} & & & & A \otimes M \ar[d]^-{\alpha_A \otimes Id_M}\\
A \otimes A \otimes M \ar@{=}[rr] \ar[d]_-{Id_A \otimes \rho} & & A \otimes A \otimes M \ar[rr]^-{\mu_A \otimes \alpha_M} \ar[d]_-{\alpha_A \otimes \rho} & & A \otimes M \ar[d]^-{\rho}\\
A \otimes M \ar[rr]_-{\alpha_A \otimes Id_M} & & A \otimes M \ar[rr]_-{\rho} & & M.
}
\]
Here the left vertical composition is $\alpha_A \otimes \rho^{1,0}$.  The right vertical and the bottom compositions are both equal to $\rho^{1,0}$.  The top rectangle is commutative by the multiplicativity of $\alpha_A$ \eqref{homassaxioms}.  The bottom left rectangle is commutative by definition, while the bottom right rectangle is commutative by the Hom-associativity of $\rho$ \eqref{eq:moduleaxiom}.
\end{proof}

Note that in Theorem \ref{thm:twisthommodule}, only the structure map $\rho$ is twisted (by $\alpha_A^n \otimes Id_M$), while $A$ and $M$ remain unchanged.  The following result is another version of the second Twisting Principle \ref{twistingprinciple} for modules over a Hom-associative algebra, where $A$, $M$, and the structure map are all twisted simultaneously.

\begin{theorem}
\label{thm2:twisthommodule}
Let $(A,\mu,\alpha_A)$ be a Hom-associative algebra and $(M,\alpha_M)$ be an $A$-module with structure map $\rho \colon A \otimes M \to M$.  For each integer $k \geq 0$, define the map
\[
\rho^{0,k} = \alpha_M^{2^k-1} \circ \rho \colon A \otimes M \to M.
\]
Then each $\rho^{0,k}$ gives the Hom-module $M^k=(M,\alpha_M^{2^k})$ the structure of an $A^k$-module, where $A^k$ is the $k$th derived Hom-associative algebra $(A,\mu^{(k)} = \alpha_A^{2^k-1} \circ \mu, \alpha_A^{2^k})$ in Theorem \ref{thm:secondtp}.
\end{theorem}

\begin{proof}
Since $\rho^{0,0} = \rho$, $\rho^{0,1} = \alpha_M \circ \rho$, $\rho^{0,k+1} = \alpha_M^{2^k} \circ \rho^{0,k}$, $M^{k+1} = (M^k)^1$, and $A^{k+1} = (A^k)^1$, by an induction argument it suffices to prove the case $k=1$.  In other words, it suffices to prove that $\rho^{0,1} = \alpha_M \circ \rho$ gives $M^1=(M,\alpha_M^2)$ the structure of an $A^1$-module, where $A^1=(A,\mu^{(1)}=\alpha_A\circ\mu,\alpha_A^2)$.

First, the multiplicativity \eqref{eq:moduleaxiom} in this case says
\[
\alpha_M^2 \circ \rho^{0,1} = \rho^{0,1} \circ (\alpha_A^2 \otimes \alpha_M^2).
\]
This is true by the following commutative diagram:
\[
\nicearrow
\xymatrix{
A \otimes M \ar[rr]^-{\rho} \ar[d]_-{\alpha_A \otimes \alpha_M} & & M \ar[rr]^-{\alpha_M} \ar[d]^-{\alpha_M} & & M \ar[d]^-{\alpha_M}\\
A \otimes M \ar[rr]^-{\rho} \ar[d]_-{\alpha_A \otimes \alpha_M} & & M \ar[d]^-{\alpha_M} & & M \ar[d]^-{\alpha_M}\\
A \otimes M \ar[rr]^-{\rho} & & M \ar[rr]^-{\alpha_M} & & M.
}
\]
The two rectangles on the left are commutative by the multiplicativity of $\alpha_A$ and $\alpha_M$ with respect to $\rho$, i.e., \eqref{eq:moduleaxiom} for the $A$-module $M$.

Next, the Hom-associativity \eqref{eq:moduleaxiom} in this case says
\[
\rho^{0,1} \circ (\alpha_A^2 \otimes \rho^{0,1}) = \rho^{0,1} \circ (\mu^{(1)} \otimes \alpha_M^2).
\]
This is true by the following commutative diagram:
\[
\nicearrow
\xymatrix{
A \otimes A \otimes M \ar[rr]^-{\mu \otimes \alpha_M} \ar[d]_-{\alpha_A \otimes \rho} & & A \otimes M \ar[rr]^-{\alpha_A \otimes \alpha_M} \ar[d]^-{\rho} & & A \otimes M \ar[d]^-{\rho}\\
A \otimes M \ar[rr]^-{\rho} \ar[d]_-{\alpha_A \otimes \alpha_M} & & M \ar[rr]^-{\alpha_M} \ar[d]^-{\alpha_M} & & M \ar[d]^-{\alpha_M}\\
A \otimes M \ar[rr]^-{\rho} & & M \ar[rr]^-{\alpha_M} & & M.
}
\]
The lower left and the upper right rectangles are commutative by the multiplicativity for the $A$-module $M$.  The upper left rectangle is commutative by the Hom-associativity for the $A$-module $M$.
\end{proof}

Starting with a given module over a Hom-associative algebra, we now combine Theorems \ref{thm:twisthommodule} and \ref{thm2:twisthommodule} to obtain a derived double-sequence of modules.

\begin{corollary}
\label{cor:twisthommodule}
Let $(A,\mu,\alpha_A)$ be a Hom-associative algebra and $(M,\alpha_M)$ be an $A$-module with structure map $\rho \colon A \otimes M \to M$.  For any integers $n,k \geq 0$, define the map
\[
\rho^{n,k} = \alpha_M^{2^k-1} \circ \rho \circ (\alpha_A^n \otimes Id_M) \colon A \otimes M \to M.
\]
Then $\rho^{n,k}$ gives the Hom-module $M^k=(M,\alpha_M^{2^k})$ the structure of an $A^k$-module, where $A^k$ is the $k$th derived Hom-associative algebra $(A,\mu^{(k)} = \alpha_A^{2^k-1} \circ \mu, \alpha_A^{2^k})$ in Theorem \ref{thm:secondtp}.
\end{corollary}

\begin{proof}
Apply Theorem \ref{thm2:twisthommodule} to the $A$-module $M$ with structure map $\rho^{n,0} = \rho \circ (\alpha_A^n \otimes Id_M)$ in Theorem \ref{thm:twisthommodule}, and observe that $\rho^{n,k} = (\rho^{n,0})^{0,k}$.
\end{proof}

If $\rho(a,m) = am$, then $\rho^{n,k}(a,m) = \alpha_M^{2^k-1}(\alpha_A^n(a)m)$.  Notice that our notations for $\rho^{*,*}$ are consistent, in the sense that $\rho^{n,0}$ has the same meaning in Corollary \ref{cor:twisthommodule} and Theorem \ref{thm:twisthommodule}.  Likewise, $\rho^{0,k}$ has the same meaning in Corollary \ref{cor:twisthommodule} and Theorem \ref{thm2:twisthommodule}.  In the context of Corollary \ref{cor:twisthommodule}, we call the $A^k$-module $M^k$ with the structure map $\rho^{n,k}$ the $(n,k)$-\textbf{derived module} of $M$.

Using Corollary \ref{cor:twisthommodule} we now prove a version of the first Twisting Principle \ref{twistingprinciple} for modules in the usual sense.

\begin{theorem}
\label{thm:twistmodule}
Let $(A,\mu_A)$ be an associative algebra and $M$ be an $A$-module in the usual sense with structure map $\rho \colon A \otimes M \to M$.  Suppose $\alpha_A \colon A \to A$ is an algebra morphism and $\alpha_M \colon M \to M$ is a linear self-map such that
\begin{equation}
\label{alphamrho}
\alpha_M \circ \rho = \rho \circ (\alpha_A \otimes \alpha_M).
\end{equation}
For any integers $n,k \geq 0$, define the map
\[
\rhoalpha^{n,k} = \alpha_M^{2^k} \circ \rho \circ (\alpha_A^{n} \otimes Id_M) \colon A \otimes M \to M.
\]
Then each $\rhoalpha^{n,k}$ gives the Hom-module $M^k_\alpha = (M,\alpha_M^{2^k})$ the structure of an $A_\beta$-module, where $\beta = \alpha_A^{2^k}$ and $A_\beta$ is the Hom-associative algebra $(A,\mu_\beta = \beta \circ \mu_A,\beta)$ in Theorem \ref{thm:hombialg}.
\end{theorem}

\begin{proof}
Using Corollary \ref{cor:twisthommodule}, it suffices to prove the case $n=k=0$, i.e., that $\rhoalpha = \rhoalpha^{0,0} = \alpha_M\circ\rho$ gives $M_\alpha = (M,\alpha_M)$ the structure of an $A_\alpha$-module, where $A_\alpha$ is the Hom-associative algebra $(A,\mu_\alpha=\alpha_A\circ\mu_A,\alpha_A)$.  Indeed, once this is proved, we can then apply Corollary \ref{cor:twisthommodule} to the $A_\alpha$-module $M_\alpha$, and observe that
\begin{equation}
\label{rhoalphank}
\begin{split}
(\rhoalpha)^{n,k} &= \alpha_M^{2^k-1} \circ \rhoalpha \circ (\alpha_A^n \otimes Id_M)\\
&= \alpha_M^{2^k} \circ \rho \circ (\alpha_A^{n} \otimes Id_M)\\
&= \rhoalpha^{n,k}
\end{split}
\end{equation}
and that $(A_\alpha)^k = A_{\alpha^{2^k}} = A_\beta$ (Remark \ref{rk:derivedhomalg}).

To prove that $M_\alpha$ is an $A_\alpha$-module with structure map $\rhoalpha = \alpha_M \circ \rho$, we need to prove the two axioms in \eqref{eq:moduleaxiom}.  First, the multiplicativity for $\rhoalpha$ says
\[
\alpha_M \circ \rho_\alpha = \rho_\alpha \circ (\alpha_A \otimes \alpha_M).
\]
This is true by the following commutative diagram:
\[
\nicearrow
\xymatrix{
A \otimes M \ar[d]_-{\alpha_A \otimes \alpha_M} \ar[rr]^-{\rho} & & M \ar[rr]^-{\alpha_M} \ar[d]^-{\alpha_M} & & M \ar[d]^-{\alpha_M}\\
A \otimes M \ar[rr]^-{\rho} & & M \ar[rr]^-{\alpha_M} & & M.
}
\]
The left rectangle is commutative by the assumption \eqref{alphamrho}.  On the other hand, the Hom-associativity for $\rhoalpha$ says
\[
\rho_\alpha \circ (\alpha_A \otimes \rho_\alpha) = \rho_\alpha \circ (\mu_\alpha \otimes \alpha_M).
\]
This is true by the following commutative diagram:
\[
\nicearrow
\xymatrix{
A \otimes A \otimes M \ar[rr]^-{\mu_A \otimes Id_M} \ar[d]_-{Id_A \otimes \rho} & & A \otimes M \ar[rr]^-{\alpha_A \otimes \alpha_M} \ar[d]_-{\rho} & & A \otimes M \ar[d]^-{\rho}\\
A \otimes M \ar[rr]^-{\rho} \ar[d]_-{\alpha_A \otimes \alpha_M} & & M \ar[rr]^-{\alpha_M} \ar[d]_-{\alpha_M} & & M \ar[d]^-{\alpha_M}\\
A \otimes M \ar[rr]^-{\rho} & & M \ar[rr]^-{\alpha_M} & & M.
}
\]
Here the top horizontal and the left vertical compositions are $\mu_\alpha \otimes \alpha_M$ and $\alpha_A \otimes \rho_\alpha$, respectively.  The top left rectangle is commutative because $M$ is an $A$-module in the usual sense.  The top right and the bottom left rectangles are commutative by the assumption \eqref{alphamrho}.  We have shown that $M_\alpha$ is an $A_\alpha$-module with structure map $\rhoalpha$.
\end{proof}

\begin{definition}
\label{def:derivedmodule}
In the setting of Theorem \ref{thm:twistmodule}, the $A_{\alpha^{2^k}}$-module $M_\alpha^k = (M,\alpha_M^{2^k})$ with the structure map $\rhoalpha^{n,k} = \alpha_M^{2^k} \circ \rho \circ (\alpha_A^{n} \otimes Id_M)$ is called the $(n,k)$-\textbf{derived module} of $M$ with respect to $(\alpha_A,\alpha_M)$, denoted $M^{n,k}_\alpha$.
\end{definition}

Theorem \ref{thm:twistmodule} says that, starting with an $A$-module $M$ in the usual sense, every pair of maps $(\alpha_A,\alpha_M)$ satisfying \eqref{alphamrho} gives rise to a derived double-sequence of modules $\{M^{n,k}_\alpha\}_{n,k\geq 0}$ over Hom-associative algebras.  If we write $\rho(a,m)$ as $am$, then the condition \eqref{alphamrho} says that $\alpha_M(am) = \alpha_A(a)\alpha_M(m)$ for all $a \in A$ and $m \in M$.  So the $(n,k)$-derived module $M_\alpha^{n,k}$ has structure map:
\[
\begin{split}
\rhoalpha^{n,k}(a,m) &= \alpha_M^{2^k}\left(\alpha^{n}_A(a)m\right)\\
&= \alpha_A^{2^k+n}(a)\alpha_M^{2^k}(m).
\end{split}
\]
Theorem \ref{thm:twistmodule} is our main tool for constructing examples of modules over Hom-quantum groups.  To use it, we need to check the condition \eqref{alphamrho}.  The following labor-saving result says that we really only need to check \eqref{alphamrho} for algebra generators of $A$ and basis elements of $M$.

\begin{proposition}
\label{prop:twistmodule}
Let $(A,\mu_A)$ be an associative algebra and $M$ be an $A$-module in the usual sense with structure map $\rho \colon A \otimes M \to M$.  Let $\alpha_A \colon A \to A$ be an algebra morphism and $\alpha_M \colon M \to M$ be a linear self-map.  Suppose that $X = \{a_i\}$ is a set of algebra generators of $A$ and that $\{m_j\}$ is a $\bk$-linear basis of $M$ such that
\begin{equation}
\label{alphamrho'}
\alpha_M(a_im_j) = \alpha_A(a_i)\alpha_M(m_j)
\end{equation}
for all $a_i \in X$ and $m_j \in M$.  Then \eqref{alphamrho} holds, i.e.,
\[
\alpha_M(am) = \alpha_A(a)\alpha_M(m)
\]
for all $a \in A$ and $m \in M$.
\end{proposition}

\begin{proof}
For an element $a \in A$, we say that $C_a$ holds if $\alpha_M(am) = \alpha_A(a)\alpha_M(m)$ for all $m \in M$.  We must show that $C_a$ holds for all $a \in A$.  Since $\{m_j\}$ is a linear basis of $M$, it follows from the assumption \eqref{alphamrho'} and the linearity of $\alpha_M$ that $C_{a_i}$ holds for each $a_i \in X$.  Since $X = \{a_i\}$ is a set of algebra generators of $A$, it remains to show the following statements:
\begin{enumerate}
\item
If $C_x$ holds and $c \in \bk$, then $C_{cx}$ holds.
\item
If $C_x$ and $C_y$ hold, then so does $C_{x+y}$.
\item
If $C_x$ and $C_y$ hold, then so does $C_{xy}$.
\end{enumerate}
The first two statements are immediate from the definition of $C_a$.  For the last statement, pick an arbitrary element $m \in M$.  Using $C_x$, $C_y$, and the fact that $M$ is an $A$-module, we have
\[
\begin{split}
\alpha_M((xy)m) &= \alpha_M(x(ym))\\
&= \alpha_A(x)\alpha_M(ym) \quad\text{by $C_x$}\\
&= \alpha_A(x)\left(\alpha_A(y)\alpha_M(m)\right) \quad\text{by $C_y$}\\
&= \left(\alpha_A(x)\alpha_A(y)\right)\alpha_M(m)\\
&= \alpha_A(xy)\alpha_M(m).
\end{split}
\]
In the last equality above, we used the assumption that $\alpha_A$ is an algebra morphism.  Thus, $C_{xy}$ holds, as desired.
\end{proof}

Using Theorem \ref{thm:twistmodule} and Proposition \ref{prop:twistmodule}, we now construct multi-parameter classes of modules over the Hom-quantum enveloping algebras $\uqsln_\al$ (Examples \ref{ex:uqg}).

\begin{example}[\textbf{Finite dimensional modules over $\uq_\alpha$}]
\label{ex:uqmodule}
Recall from Example \ref{ex:uqg} the Hom-quantum enveloping algebra $\uq_\al$, where the bialgebra morphism $\al \colon \uq \to \uq$ is defined on the algebra generators by
\begin{equation}
\label{alphaefk}
\al(E) = \lambda E, \quad \al(F) = \lambda^{-1}F, \quad\text{and}\quad \al(K^{\pm 1}) = K^{\pm 1}
\end{equation}
for a fixed invertible scalar $\lambda \in \bC$.  In this example, we construct four-parameter, uncountable classes of finite-dimensional derived modules over the Hom-quantum enveloping algebras of $\sltwo$ by applying Theorem \ref{thm:twistmodule} to finite-dimensional simple $\uq$-modules.  The references for these simple $\uq$-modules are \cite{lusztig,rosso88}.

Assume that $q \in \bC \setminus \{0\}$ is not a root of unity.  For each integer $n \geq 0$ and $\epsilon \in \{\pm 1\}$, there is an $(n+1)$-dimensional simple $\uq$-module $V(\epsilon,n)$.  Let $\{v_i\}_{i=0}^n$ be a basis of $V(\epsilon,n)$.  Then the $\uq$-module action $\rho$ on $V(\epsilon,n)$ is determined by
\begin{equation}
\label{kefv}
Kv_i = \epsilon q^{n-2i}v_i,\quad
Ev_i = \epsilon [n-i+1]_q v_{i-1},\quad\text{and}\quad
Fv_i = [i+1]_q v_{i+1},
\end{equation}
where $v_{-1} = 0 = v_{n+1}$ and the $q$-integers were as defined in \eqref{qinteger}.  To use Theorem \ref{thm:twistmodule} on the simple $\uq$-module $V(\epsilon,n)$, we need linear maps on $V(\epsilon,n)$ such that \eqref{alphamrho} is satisfied.  Such maps can be constructed as follows.

Pick any scalar $\xi \in \bC$, and define $\ax \colon V(\epsilon,n) \to V(\epsilon,n)$ by setting
\begin{equation}
\label{alphaxi}
\ax(v_i) = \xi\lambda^{-i}v_i
\end{equation}
for all $i$.  We claim that $\al$ and $\ax$ satisfy \eqref{alphamrho}, i.e.,
\begin{equation}
\label{axuv}
\ax(uv) = \al(u)\ax(v)
\end{equation}
for all $u \in \uq$ and $v \in V(\epsilon,n)$.  By Proposition \ref{prop:twistmodule} we only need to check \eqref{axuv} for the algebra generators $u \in \{E,F,K^{\pm 1}\}$ of $\uq$ and for the basis elements $v \in \{v_i\}_{i=0}^n$ of $V(\epsilon,n)$.  When $u = K^{\pm 1}$, we have
\begin{equation}
\label{alphakvi}
\begin{split}
\ax(K^{\pm 1}v_i) &= \ax((\epsilon q^{n-2i})^{\pm 1} v_i)\\
&= \xi\lambda^{-i}(\epsilon q^{n-2i})^{\pm 1} v_i\\
&= K^{\pm 1}(\xi\lambda^{-i}v_i)\\
&= \al(K^{\pm 1})\ax(v_i).
\end{split}
\end{equation}
Likewise, when $u = E$, we have
\[
\begin{split}
\ax(Ev_i) &= \ax(\epsilon [n-i+1]_q v_{i-1})\\
&= \epsilon [n-i+1]_q \xi \lambda^{-(i-1)}v_{i-1}\\
&= (\lambda E)(\xi \lambda^{-i}v_i)\\
&= \al(E)\ax(v_i).
\end{split}
\]
Finally, when $u = F$, we have
\begin{equation}
\label{axfv}
\begin{split}
\ax(Fv_i) &= \ax([i+1]_q v_{i+1})\\
&= [i+1]_q \xi \lambda^{-(i+1)}v_{i+1}\\
&= (\lambda^{-1}F)(\xi \lambda^{-i}v_i)\\
&= \al(F)\ax(v_i).
\end{split}
\end{equation}
Therefore, \eqref{axuv} is satisfied.

By Theorem \ref{thm:twistmodule}, for any integers $r,k \geq 0$, there is an $(n+1)$-dimensional derived $\uqbeta$-module $V(\epsilon,n)_\alpha^{r,k} = (V(\epsilon,n),\ax^{2^k})$, where $\beta = \al^{2^k} = \alpha_{\lambda^{2^k}}$, with structure map
\[
\rhoalpha^{r,k} = \ax^{2^k} \circ \rho \circ (\al^r \otimes Id_{V(\epsilon,n)}) \colon \uq \otimes V(\epsilon,n) \to V(\epsilon,n).
\]
More explicitly, using \eqref{alphaefk}, \eqref{kefv}, and \eqref{alphaxi}, we have
\[
\begin{split}
\rhoalpha^{r,k}(K^{\pm 1}, v_i) &= (\epsilon q^{n-2i})^{\pm 1} (\xi \lambda^{-i})^{2^k}v_i,\\
\rhoalpha^{r,k}(E, v_i) &= \epsilon [n-i+1]_q \xi^{2^k} \lambda^{r-2^k(i-1)} v_{i-1},\quad\text{and}\\
\rhoalpha^{r,k}(F, v_i) &= [i+1]_q \xi^{2^k} \lambda^{-r-2^k(i+1)} v_{i+1}.
\end{split}
\]
Given $n \geq 0$, we have constructed an uncountable, four-parameter family
\[
\{V(\epsilon,n)_\alpha^{r,k} \colon \lambda \in \bC \setminus \{0\},\, \xi \in \bC,\, r,k \geq 0\}
\]
of $(n+1)$-dimensional derived modules over the Hom-quantum enveloping algebras $\uqbeta$.  We recover the original $(n+1)$-dimensional simple $\uq$-module $V(\epsilon,n)$ by taking $\lambda = \xi = 1$, since in this case both $\beta$ and $\ax$ are the identity maps.
\qed
\end{example}

\begin{example}[\textbf{Hom-Verma modules over $\uq_\alpha$}]
\label{ex:verma}
In this example, we construct a four-parameter, uncountable class of infinite-dimensional derived modules over the Hom-quantum enveloping algebras of $\sltwo$ by applying Theorem \ref{thm:twistmodule} to the Verma modules over $\uq$.  The references for the Verma modules are \cite{cp} (Chapter 10) and \cite{kassel} (Chapter VI).

We keep the notations of the previous example, so we are still working over $\bC$, and $q \in \bC \setminus \{0\}$ is not a root of unity.  Pick any non-zero scalar $\eta \in \bC$, and consider the infinite-dimensional $\bC$-module $\mqeta$ with a basis $\{v_i \colon i \geq 0\}$.  Then $\mqeta$ is a $\uq$-module, called the \emph{Verma module}, where the $\uq$-module action $\rho$ is determined by
\begin{equation}
\label{vermaaction}
\begin{split}
K^{\pm 1}v_i &= (\eta q^{-2i})^{\pm 1}v_i,\\
Fv_i &= [i+1]_q v_{i+1},\\
Ev_{i+1} &= \frac{q^{-i}\eta - q^i\eta^{-1}}{q - q^{-1}}v_i,\quad\text{and}\quad Ev_0 = 0.\\
\end{split}
\end{equation}
Define $\ax \colon \mqeta \to \mqeta$ as in \eqref{alphaxi}, i.e., $\ax(v_i) = \xi \lambda^{-i}v_i$.  We claim that \eqref{alphamrho} holds in this situation.  By Proposition \ref{prop:twistmodule} it suffices to check \eqref{alphamrho} for the algebra generators $\{E,F,K^{\pm 1}\}$ of $\uq$ and for the basis elements $v_i \in \mqeta$.  The necessary computation is similar to the previous example.  We have
\[
\begin{split}
\ax(Ev_{i+1}) &= \frac{q^{-i}\eta - q^i\eta^{-1}}{q - q^{-1}} \ax(v_i)\\
&= \xi \lambda^{-i} \frac{q^{-i}\eta - q^i\eta^{-1}}{q - q^{-1}}v_i\\
&= \xi \lambda^{-i} (Ev_{i+1})\\
&= (\lambda E)(\xi \lambda^{-(i+1)}v_{i+1})\\
&= \al(E)\ax(v_{i+1}).
\end{split}
\]
The computation proving $\ax(Fv_i) = \al(F)\ax(v_i)$ is identical to \eqref{axfv}.  Finally, to prove $\ax(K^{\pm 1}v_i) = \al(K^{\pm 1})\ax(v_i)$, one uses the computation \eqref{alphakvi} with $\epsilon q^{n-2i}$ replaced by $\eta q^{-2i}$.

Therefore, by Theorem \ref{thm:twistmodule}, for any integers $r,k \geq 0$, there is an infinite-dimensional $\uqbeta$-module $\mqeta^{r,k}_\alpha = (\mqeta,\ax^{2^k})$, where $\beta = \al^{2^k} = \alpha_{\lambda^{2^k}}$, with structure map
\[
\rhoalpha^{r,k} = \ax^{2^k} \circ \rho \circ (\al^r \otimes Id) \colon \uq \otimes \mqeta \to \mqeta.
\]
More explicitly, using \eqref{alphaefk}, \eqref{alphaxi}, and \eqref{vermaaction}, we have
\[
\begin{split}
\rhoalpha^{r,k}(K^{\pm 1},v_i) &= (\eta q^{-2i})^{\pm 1}(\xi \lambda^{-i})^{2^k}v_i,\\
\rhoalpha^{r,k}(E,v_{i+1}) &= \frac{q^{-i}\eta - q^i\eta^{-1}}{q - q^{-1}} \xi^{2^k} \lambda^{r-2^ki} v_i, \quad\text{and}\\
\rhoalpha^{r,k}(F,v_i) &= [i+1]_q \xi^{2^k} \lambda^{-r-2^k(i+1)} v_{i+1},
\end{split}
\]
where $v_{-1} = 0$.  Given a non-zero scalar $\eta \in \bC$, we have constructed an uncountable, four-parameter family
\[
\{\mqeta^{r,k}_\alpha \colon \lambda \in \bC \setminus\{0\},\, \xi \in \bC,\, r,k \geq 0\}
\]
of infinite-dimensional derived modules over the Hom-quantum enveloping algebras $\uqbeta$.  We recover the original Verma module $\mqeta$ by taking $\lambda = \xi = 1$.
\qed
\end{example}

\begin{example}[\textbf{Finite dimensional modules over $\uqsln_\alpha$}]
\label{ex:slnmodule}
Assume that $q \in \bC \setminus \{0\}$ is not a root of unity, and fix an integer $n \geq 2$.  In this example, we construct an uncountable, $(n+2)$-parameter family of $n$-dimensional derived $\uqsln_\alpha$-modules by twisting some standard $\uqsln$-modules \cite{hk} (Example 4.2.7).  Here $\uqsln_\alpha$ are the Hom-quantum enveloping algebras of $\sln$ constructed in Example \ref{ex:uqg}, whose Cartan matrix is described in \eqref{slcartan}. In particular, the quantum enveloping algebra $\uqsln$ is generated as a unital algebra by $\{E_i,F_i,K_i^{\pm 1}\}_{i=1}^{n-1}$ with relations \eqref{uqgrelations}.  The bialgebra morphism $\al \colon \uqsln \to \uqsln$ is defined in \eqref{alphalambdauqg}: $\al(K_i^{\pm 1}) = K_i^{\pm 1}$,
$\al(E_i) = \lambda_i E_i$, and $\al(F_i) = \lambda_i^{-1}F_i$.

Let $V_n$ be an $n$-dimensional $\bC$-module with a basis $\{v_i\}_{i=1}^n$.  It is a $\uqsln$-module whose structure map $\rho$ is determined by
\begin{equation}
\label{uqslnvn}
\begin{split}
E_i v_j &= \delta_{i+1,j}v_i, \quad
F_i v_j = \delta_{ij}v_{i+1},\quad\text{and}\\
K_i^{\pm 1} v_j &=
\begin{cases}
q^{\pm 1}v_i & \text{if $j=i$},\\
q^{\mp 1}v_{i+1} & \text{if $j=i+1$},\\
v_j & \text{otherwise}.
\end{cases}
\end{split}
\end{equation}
Fix a scalar $\xi \in \bC$.  Consider the linear map $\ax \colon V_n \to V_n$ defined by
\begin{equation}
\label{alphaxivn}
\ax(v_i) =
\xi(\lambda_1 \cdots \lambda_{i-1})^{-1}v_i,
\end{equation}
where, when $i=1$, the empty product $\lambda_1 \cdots \lambda_{i-1}$ is taken to mean $1$.  We claim that \eqref{alphamrho} holds, i.e., $\ax(uv) = \al(u)\ax(v)$ for all $u \in \uqsln$ and $v \in V_n$.  By Proposition \ref{prop:twistmodule} it suffices to check this for the algebra generators $u \in \{E_i,F_i,K_i^{\pm 1}\}_{i=1}^{n-1}$ of $\uqsln$ and for the basis elements $v \in \{v_i\}_{i=1}^n$ of $V_n$.  This is easily checked on a case-by-case basis as in the previous two examples.  In fact, we only need to consider the pairs $(u,v)$ where $uv\not= 0$ in \eqref{uqslnvn}.  We have
\[
\begin{split}
\ax(E_i v_{i+1}) &= \ax(v_i)\\
&= \xi(\lambda_1 \cdots \lambda_{i-1})^{-1}v_i\\
&= (\lambda_i E_i)(\xi(\lambda_1 \cdots \lambda_i)^{-1}v_{i+1})\\
&= \al(E_i)\ax(v_{i+1}).
\end{split}
\]
Likewise, we have
\[
\begin{split}
\ax(F_iv_i) &= \ax(v_{i+1})\\
&= \xi (\lambda_1 \cdots \lambda_i)^{-1}v_{i+1}\\
&= (\lambda_i^{-1}F_i)(\xi(\lambda_1 \cdots \lambda_{i-1})^{-1}v_i)\\
&= \al(F_i)\ax(v_i)
\end{split}
\]
Finally, for the case $u = K_i^{\pm 1}$ with a fixed $i \in \{1, \ldots , n-1\}$, set
\[
p_{ij} = \begin{cases}
q^{\pm 1} & \text{if $j=i$},\\
q^{\mp 1} & \text{if $j=i+1$},\\
1 & \text{otherwise}.
\end{cases}
\]
Then we have $K_i^{\pm 1}v_j = p_{ij}v_j$ and
\[
\begin{split}
\ax(K_i^{\pm 1}v_j) &= p_{ij} \xi (\lambda_1 \cdots \lambda_{j-1})^{-1}v_j\\
&= \al(K_i^{\pm 1})\ax(v_j).
\end{split}
\]

Therefore, by Theorem \ref{thm:twistmodule}, for any integers $r,k \geq 0$, there is an $n$-dimensional derived $\uqsln_\beta$-module $(V_n)^{r,k}_\alpha = (V_n,\ax^{2^k})$, where $\beta = \al^{2^k}$, with structure map
\[
\rhoalpha^{r,k} = \ax^{2^k} \circ \rho \circ (\al^r \otimes Id) \colon \uqsln \otimes V_n \to V_n.
\]
More explicitly, using \eqref{alphalambdauqg}, \eqref{uqslnvn}, and \eqref{alphaxivn}, we have
\[
\begin{split}
\rhoalpha^{r,k}(E_i,v_j) &= \delta_{i+1,j} \xi^{2^k} (\lambda_1 \cdots \lambda_{i-1})^{-2^k}\lambda_i^r v_i,\\
\rhoalpha^{r,k}(F_i,v_j) &= \delta_{ij}\xi^{2^k} (\lambda_1 \cdots \lambda_{i-1})^{-2^k}\lambda_i^{-r-2^k}v_{i+1},\quad\text{and}\\
\rhoalpha^{r,k}(K_i^{\pm 1}, v_j) &=
p_{ij}\xi^{2^k} (\lambda_1 \cdots \lambda_{j-1})^{-2^k} v_j.
\end{split}
\]
In summary, given $n \geq 2$, we have constructed an uncountable, $(n+2)$-parameter family
\[
\{(V_n)^{r,k}_\alpha \colon \lambda_1, \ldots , \lambda_{n-1} \in \bC \setminus\{0\},\, \xi \in \bC,\, r,k \geq 0\}
\]
of $n$-dimensional derived modules over the Hom-quantum enveloping algebras $\uqsln_\beta$.  We recover the original $\uqsln$-module $V_n$ by taking $\lambda_1 = \cdots = \lambda_{n-1} = \xi = 1$.
\qed
\end{example}

Corollary \ref{cor:twisthommodule} and Theorem \ref{thm:twistmodule} can be readily dualized by inverting the arrows and replacing $\mu$ by $\Delta$ in the various commutative diagrams in their proofs.  Therefore, we omit the proofs of the following two results, which are dual to Corollary \ref{cor:twisthommodule} and Theorem \ref{thm:twistmodule}, respectively.

\begin{theorem}
\label{thm:twisthomco}
Let $(C,\Delta,\alpha_C)$ be a Hom-coassociative coalgebra and $(M,\alpha_M)$ be a $C$-comodule with structure map $\rho \colon M \to C \otimes M$.  For any integers $n,k \geq 0$, define the map
\[
\rho^{n,k} = (\alpha_C^n \otimes Id_M) \circ \rho \circ \alpha_M^{2^k-1} \colon M \to C \otimes M.
\]
Then each $\rho^{n,k}$ gives the Hom-module $M^k=(M,\alpha_M^{2^k})$ the structure of a $C^k$-comodule, where $C^k$ is the $k$th derived Hom-coassociative coalgebra $(C,\Delta^{(k)}=\Delta\circ\alpha_C^{2^k-1},\alpha_C^{2^k})$ in Theorem \ref{thm:secondtp}.
\end{theorem}

\begin{theorem}
\label{thm:twistco}
Let $(C,\Delta_C)$ be a coassociative coalgebra and $M$ be a $C$-comodule in the usual sense with structure map $\rho \colon M \to C \otimes M$.  Suppose $\alpha_C \colon C \to C$ is a coalgebra morphism and $\alpha_M \colon M \to M$ is a linear self-map such that
\[
\rho \circ \alpha_M = (\alpha_C \otimes \alpha_M)\circ\rho.
\]
For any integers $n,k \geq 0$, define the map
\[
\rhoalpha^{n,k} = (\alpha_C^{n} \otimes Id_M) \circ \rho \colon \alpha_M^{2^k} \colon M \to C \otimes M.
\]
Then each $\rhoalpha^{n,k}$ gives the Hom-module $M_\alpha^k = (M,\alpha_M^{2^k})$ the structure of a $C_\beta$-comodule, where $\beta = \alpha_C^{2^k}$ and $C_\beta$ is the Hom-coassociative coalgebra $(C,\Delta_\beta=\Delta_C \circ \beta,\beta)$ in Theorem \ref{thm:hombialg}.
\end{theorem}

\section{Module Hom-algebras}
\label{sec:mha}

Module-algebras play an important role in quantum group theory.  In this section we study the Hom version of a module-algebra, called a module Hom-algebra, which consists of a Hom-associative algebra with a suitable module action by a Hom-bialgebra.  In Theorem \ref{thm:mhachar} we provide an alternative characterization of the module Hom-algebra axiom \eqref{modhomalg}.  Then we establish the two Twisting Principles \ref{twistingprinciple} for module Hom-algebras (Theorems \ref{thm:twistmha} and \ref{thm:twistma}).  Examples of module Hom-algebras related to Hom-quantum geometry on the Hom-quantum planes (Example \ref{ex:hqspace}) are considered in Examples \ref{ex:hqg} and \ref{ex2:hqg}.

Before we give the definition of a module Hom-algebra, let us first recall the definition of a module-algebra.  Let $H$ be a bialgebra and $A$ be an associative algebra.  Then an \textbf{$H$-module-algebra} structure on $A$ consists of an $H$-module structure $\rho \colon H \otimes A \to A$, written $\rho(x,a) = xa$, such that the \textbf{module-algebra axiom}
\begin{equation}
\label{modalgaxiom}
x(ab) = \sum_{(x)} (x_1a)(x_2b)
\end{equation}
is satisfied for all $x \in H$ and $a,b \in A$.  Recall that we use Sweedler's notation $\Delta(x) = \sum_{(x)} x_1 \otimes x_2$ for comultiplication.  In element-free form, the module-algebra axiom is
\[
\rho \circ (Id_H \otimes \mu_A) = \mu_A \circ \rho^{\otimes 2} \circ (2~3) \circ (\Delta_H \otimes Id_A \otimes Id_A),
\]
where $(2~3) = Id_H \otimes \tau_{H,A} \otimes Id_A$ and $\tau_{H,A} \colon H \otimes A \to A \otimes H$ is the twist isomorphism.

We already defined Hom-associative algebras, Hom-bialgebras (Definition \ref{def:homas}), and modules over a Hom-associative algebra (Definition \ref{def:hommodule}).  So to define the Hom version of a module-algebra, we need a suitable Hom-type analog of the module-algebra axiom.

\begin{definition}
\label{def:mha}
Let $(H,\mu_H,\Delta_H,\alpha_H)$ be a Hom-bialgebra and $(A,\mu_A,\alpha_A)$ be a Hom-associative algebra.  A \textbf{$H$-module Hom-algebra} structure on $A$ consists of an $H$-module structure $\rho \colon H \otimes A \to A$ (as in Definition \ref{def:hommodule}) such that the \textbf{module Hom-algebra axiom}
\begin{equation}
\label{modhomalg}
\alpha^2_H(x)(ab) = \sum_{(x)} (x_1a)(x_2b)
\end{equation}
is satisfied for all $x \in H$ and $a,b \in A$, where $\rho(x,a) = xa$.
\end{definition}
In element-free form, the module Hom-algebra axiom \eqref{modhomalg} is
\begin{equation}
\label{mha'}
\rho \circ (\alpha_H^2 \otimes \mu_A) = \mu_A \circ \rho^{\otimes 2} \circ (2~3) \circ (\Delta_H \otimes Id_A \otimes Id_A).
\end{equation}
An $H$-module Hom-algebra with $\alpha_H = Id_H$ and $\alpha_A = Id_A$ is exactly an $H$-module-algebra in the usual sense.  We will prove the Twisting Principles \ref{twistingprinciple} for module Hom-algebras.  Before that let us give a more conceptual characterization of the module Hom-algebra axiom \eqref{modhomalg}.  To do that we need the following preliminary result.

\begin{proposition}
\label{prop:hombimodule}
Let $(H,\mu,\Delta,\alpha_H)$ be a Hom-bialgebra, and let $(M,\alpha_M)$ and $(N,\alpha_N)$ be $H$-modules with structure maps $\rho_M$ and $\rho_N$, respectively.  Then $M \otimes N$ is an $H$-module with structure map
\begin{equation}
\label{rhomn}
\rho_{MN} = (\rho_M \otimes \rho_N) \circ (2~3) \circ (\Delta \otimes Id_M \otimes Id_N) \colon H \otimes M \otimes N \to M \otimes N,
\end{equation}
where $(2~3) =  Id_H \otimes \tau_{H,M} \otimes Id_N$ and $\tau_{H,M} \colon H \otimes M \to M \otimes H$ is the twist isomorphism.
\end{proposition}

\begin{proof}
We need to prove the two conditions in \eqref{eq:moduleaxiom} for $\rho_{MN}$ and $\alpha_M \otimes \alpha_N$.  First, the multiplicativity condition for $\rho_{MN}$ says
\[
(\alpha_M \otimes \alpha_N) \circ \rho_{MN} = \rho_{MN} \circ (\alpha_H \otimes \alpha_M \otimes \alpha_N).
\]
This is true by the following commutative diagram:
\[
\nicearrow
\xymatrix{
H \otimes M \otimes N \ar[rrr]^-{\alpha_H \otimes \alpha_M \otimes \alpha_N} \ar[d]_-{\Delta \otimes Id_M \otimes Id_N} & & & H \otimes M \otimes N \ar[d]^-{\Delta \otimes Id_M \otimes Id_N}\\
H^{\otimes 2} \otimes M \otimes N \ar[rrr]^-{\alpha_H^{\otimes 2} \otimes \alpha_M \otimes \alpha_N} \ar[d]_-{(2~3)} & & & H^{\otimes 2} \otimes M \otimes N \ar[d]^-{(2~3)}\\
H \otimes M  \otimes H \otimes N \ar[rrr]^-{\alpha_H \otimes \alpha_M \otimes \alpha_H \otimes \alpha_N} \ar[d]_-{\rho_M \otimes \rho_N} & & & H \otimes M  \otimes H \otimes N \ar[d]^-{\rho_M \otimes \rho_N}\\
M \otimes N \ar[rrr]^-{\alpha_M \otimes \alpha_N} & & & M \otimes N.
}
\]
The top rectangle is commutative by the comultiplicativity in $H$.  The middle rectangle is commutative by definition.  The bottom rectangle is commutative by the multiplicativity for $\rho_M$ and $\rho_N$.

Next, the Hom-associativity for $\rho_{MN}$ says
\begin{equation}
\label{homassrhomn}
\rho_{MN} \circ (\alpha_H \otimes \rho_{MN}) = \rho_{MN} \circ (\mu \otimes \alpha_M \otimes \alpha_N).
\end{equation}
To prove \eqref{homassrhomn}, we use the abbreviations:
\[
\begin{split}
\rho' &= (\rho_M \otimes \rho_N) \circ (2~3) \colon H^{\otimes 2} \otimes M \otimes N \to M \otimes N,\\
\mu'& = \mu^{\otimes 2} \circ (2~3) \colon H^{\otimes 4} \to H^{\otimes 2},\\
P &= M \otimes N, \quad\text{and}\\
\alpha_P &= \alpha_M \otimes \alpha_N.
\end{split}
\]
With these notations, we have
\[
\begin{split}
\rho_{MN} &= (\rho_M \otimes \rho_N) \circ (2~3) \circ (\Delta \otimes Id_M \otimes Id_N)\\
&= \rho' \circ (\Delta \otimes Id_P)
\end{split}
\]
and
\begin{equation}
\label{deltamuh}
\begin{split}
\Delta \circ \mu &= \mu^{\otimes 2} \circ (2~3) \circ \Delta^{\otimes 2}\\
&= \mu' \circ \Delta^{\otimes 2}
\end{split}
\end{equation}
by \eqref{def:hombi}.  The commutativity of the following diagram proves \eqref{homassrhomn}:
\[
\nicearrow
\xymatrix{
H^{\otimes 2} \otimes P \ar[rr]^-{\alpha_H \otimes \Delta \otimes Id_P} \ar[ddrr]^-{\Delta^{\otimes 2} \otimes Id_P} \ar[dddd]_-{\mu \otimes \alpha_P} & & H^{\otimes 3} \otimes P \ar[rr]^-{Id_H \otimes \rho'} & & H \otimes P \ar[dd]^-{\Delta \otimes Id_P}\\
& & & & \\
& & H^{\otimes 4} \otimes P \ar[rr]^-{\alpha^{\otimes 2} \otimes \rho'} \ar[dd]_-{\mu' \otimes \alpha_P} & & H^{\otimes 2} \otimes P \ar[dd]^-{\rho'}\\
& & & & \\
H \otimes P \ar[rr]_-{\Delta \otimes Id_P} & & H^{\otimes 2} \otimes P \ar[rr]_-{\rho'} & & P.
}
\]
The composition along the top and then the right edges of the big square is the left-hand side of \eqref{homassrhomn}.  The composition along the left and then the bottom edges of the big square is the right-hand side of \eqref{homassrhomn}.  The left trapezoid is commutative by \eqref{deltamuh}.  The top trapezoid is commutative by the comultiplicativity in $H$.  The bottom right square is commutative by the Hom-associativity \eqref{eq:moduleaxiom} for $\rho_M$ and $\rho_N$.
\end{proof}

Proposition \ref{prop:hombimodule} has an obvious analog for comodules.  Its proof is exactly dual to that of Proposition \ref{prop:hombimodule}.  Indeed, to prove the two conditions in \eqref{eq:comoduleaxiom} for the tensor product of two comodules, one simply inverts the arrows and interchanges $\mu$ and $\Delta$ in the two commutative diagrams in the proof of Proposition \ref{prop:hombimodule}.  Thus, we will omit the proof of the following dual result.

\begin{proposition}
Let $(H,\mu,\Delta,\alpha_H)$ be a Hom-bialgebra, and let $(M,\alpha_M)$ and $(N,\alpha_N)$ be $H$-comodules with structure maps $\rho_M$ and $\rho_N$, respectively.  Then $M \otimes N$ is an $H$-comodule with structure map
\[
\rho_{MN} = (\mu \otimes Id_M \otimes Id_N) \circ (2~3) \circ (\rho_M \otimes \rho_N) \colon M \otimes N \to H \otimes M \otimes N,
\]
where $(2~3) =  Id_H \otimes \tau_{M,H} \otimes Id_N$.
\end{proposition}

In the following result, we give an alternative, more conceptual characterization of a module Hom-algebra (Definition \ref{def:mha}), generalizing a similar characterization of module-algebras.

\begin{theorem}
\label{thm:mhachar}
Let $H = (H,\mu_H,\Delta_H,\alpha_H)$ be a Hom-bialgebra, $A = (A,\mu_A,\alpha_A)$ be a Hom-associative algebra, and $\rho \colon H \otimes A \to A$ be an $H$-module structure on $A$.  Then the module Hom-algebra axiom \eqref{modhomalg} is satisfied if and only if $\mu_A \colon A^{\otimes 2} \to A$ is a morphism of $H$-modules, in which $A^{\otimes 2}$ and $A$ are given the $H$-module structure maps $\rho_{AA}$ \eqref{rhomn} and $\rho^{2,0}$ \eqref{rhon0}, respectively.
\end{theorem}

\begin{proof}
Since $A$ is a Hom-associative algebra, we already know that $\mu_A$ is a morphism of Hom-modules.  In the module Hom-algebra axiom \eqref{modhomalg}, the left-hand side is
\[
\alpha_H^2(x)(ab) = \left(\rho^{2,0} \circ (Id_H \otimes \mu_A)\right)(x \otimes a \otimes b).
\]
Likewise, the right-hand side in \eqref{modhomalg} is
\[
\sum_{(x)} (x_1a)(x_2b) = \left(\mu_A \circ \rho_{AA}\right)(x \otimes a \otimes b).
\]
Therefore, the module Hom-algebra axiom \eqref{modhomalg} is equivalent to
\[
\rho^{2,0} \circ (Id_H \otimes \mu_A) = \mu_A \circ \rho_{AA}.
\]
This equality is equivalent to $\mu_A \colon A^{\otimes 2} \to A$ being a morphism of $H$-modules (see \eqref{eq:modmorphism}), provided $A^{\otimes 2}$ and $A$ are equipped with the $H$-module structure maps $\rho_{AA}$ and $\rho^{2,0}$, respectively.
\end{proof}

We now prove the Twisting Principles \ref{twistingprinciple} for module Hom-algebras along the lines of Corollary \ref{cor:twisthommodule} and Theorem \ref{thm:twistmodule}.  We begin with the following version of the second Twisting Principle, which says that every module Hom-algebra gives rise to a derived double-sequence of module Hom-algebras.

\begin{theorem}
\label{thm:twistmha}
Let $H = (H,\mu_H,\Delta_H,\alpha_H)$ be a Hom-bialgebra, $A = (A,\mu_A,\alpha_A)$ be a Hom-associative algebra, and $\rho \colon H \otimes A \to A$ be an $H$-module Hom-algebra structure on $A$.  For any integers $n,k \geq 0$, define the map
\[
\rho^{n,k} = \alpha_A^{2^k-1} \circ \rho \circ (\alpha_H^n \otimes Id_A) \colon H \otimes A \to A.
\]
Then $\rho^{n,k}$ gives $A^k$ the structure of an $H^k$-module Hom-algebra, where $A^k$ is the $k$th derived Hom-associative algebra $(A,\mu_A^{(k)} = \alpha_A^{2^k-1} \circ \mu_A, \alpha_A^{2^k})$ and $H^k$ is the $k$th derived Hom-bialgebra $(H,\mu_H^{(k)} = \alpha_H^{2^k-1} \circ \mu_H,\Delta_H^{(k)} = \Delta_H \circ \alpha_H^{2^k-1},\alpha_H^{2^k})$ in Theorem \ref{thm:secondtp}.
\end{theorem}

\begin{proof}
Exactly as in Corollary \ref{cor:twisthommodule}, since $\rho^{n,k} = (\rho^{n,0})^{0,k}$, it suffices to prove the cases $\rho^{n,0}$ and $\rho^{0,k}$.  Furthermore, as in Theorems \ref{thm:twisthommodule} and \ref{thm2:twisthommodule}, by an induction argument it suffices to prove the two cases $\rho^{1,0}$ and $\rho^{0,1}$.  Since Corollary \ref{cor:twisthommodule} already tells us that $A^k$ is an $H^k$-module with structure map $\rho^{n,k}$, in each of the two cases $\rho^{1,0}$ and $\rho^{0,1}$ we only need to prove the module Hom-algebra axiom \eqref{modhomalg}, or equivalently \eqref{mha'}.

For the case $\rho^{1,0} = \rho \circ (\alpha_H \otimes Id_A)$, first note that $A^0 = A$ and $H^0 = H$.  Thus, the module Hom-algebra axiom \eqref{mha'} for $\rho^{1,0}$ says
\begin{equation}
\label{mharho10}
\rho^{1,0} \circ (\alpha_H^2 \otimes \mu_A) = \mu_A \circ (\rho^{1,0})^{\otimes 2} \circ (2~3) \circ (\Delta_H \otimes Id_A \otimes Id_A).
\end{equation}
This is true by the following commutative diagram, where we abbreviate $A^{\otimes 2}$ to $B$:
\[
\nicearrow
\xymatrix{
H \otimes B \ar[rr]^-{\Delta \otimes Id_B} \ar[ddrr]^-{\alpha_H \otimes Id_B} \ar[dddd]_-{\alpha_H^2 \otimes \mu_A} & & H^{\otimes 2} \otimes B \ar[rrr]^-{(\alpha_H \otimes Id_A)^{\otimes 2} \circ (2~3)} \ar[ddrr]^-{\alpha_H^{\otimes 2} \otimes Id_B} & & & H \otimes A \otimes H \otimes A \ar[dd]^-{\rho^{\otimes 2}}\\
& & & & & \\
& & H \otimes B \ar[rr]^-{\Delta \otimes Id_B} \ar[dd]^-{\alpha_H^2 \otimes \mu_A} & & H^{\otimes 2} \otimes B \ar[r]^-{\rho^{\otimes 2}\circ (2~3)} & B \ar[dd]^-{\mu_A}\\
& & & & & \\
H \otimes A \ar[rr]^-{\alpha_H \otimes Id_A} & & H \otimes A \ar[rrr]^-{\rho} & & & A.
}
\]
The composition along the left and then the bottom edges is the left-hand side in \eqref{mharho10}.  The composition along the top and then the right edges is the right-hand side in \eqref{mharho10}.  The left trapezoid is commutative because both compositions are equal to $\alpha_H^3 \otimes \mu_A$.  The top parallelogram is commutative by the comultiplicativity in $H$ \eqref{homcoassaxioms}.  The top right trapezoid is commutative by definition.  The bottom right rectangle is commutative by the $H$-module Hom-algebra axiom of $A$ \eqref{mha'}.

For the case $\rho^{0,1} = \alpha_A \circ \rho$, the $H^1$-module Hom-algebra axiom for $A^1$ \eqref{mha'} says
\begin{equation}
\label{mharho01}
\rho^{0,1} \circ (\alpha_H^4 \otimes \mu_A^{(1)}) = \mu_A^{(1)} \circ (\rho^{0,1})^{\otimes 2} \circ (2~3) \circ (\Delta_H^{(1)} \otimes Id_A \otimes Id_A).
\end{equation}
With $B = A^{\otimes 2}$, the following commutative diagram proves \eqref{mharho01}:
\[
\nicearrow
\xymatrix{
H \otimes B \ar[rr]^-{(\Delta \circ \alpha_H) \otimes Id_B} \ar[ddrr]^-{\alpha_H^3 \otimes \mu_A} \ar[dd]_-{Id_H \otimes \mu_A} & & H^{\otimes 2} \otimes B \ar[rr]^-{\rho^{\otimes 2} \circ (2~3)} & & B \ar[rr]^-{\alpha_A^{\otimes 2}} \ar[dd]^-{\mu_A} & & B \ar[dd]^-{\mu_A}\\
& & & & & & \\
H \otimes A \ar[dd]_-{\alpha_H^4 \otimes \alpha_A} & & H \otimes A \ar[rr]^-{\rho} \ar[ddll]^-{\alpha_H \otimes \alpha_A} & & A \ar[rr]^-{\alpha_A} \ar[dd]^-{\alpha_A} & & A \ar[dd]^-{\alpha_A}\\
& & & & & & \\
H \otimes A \ar[rrrr]^-{\rho} & & & & A \ar[rr]^-{\alpha_A} & & A.
}
\]
The composition along the left and then the bottom edges is the left-hand side in \eqref{mharho01}.  The composition along the top and then the right edges is the right-hand side in \eqref{mharho01}.  The left triangle is commutative because both compositions are equal to $\alpha_H^4 \otimes (\alpha_A \circ \mu_A)$.  The top trapezoid is commutative by the $H$-module Hom-algebra axiom \eqref{mha'} of $A$.  The bottom trapezoid is commutative by the multiplicativity of the $H$-module structure map $\rho$ \eqref{eq:moduleaxiom}.  The top right square is commutative by the multiplicativity in $A$ \eqref{homassaxioms}.
\end{proof}

Next we have the following version of the first Twisting Principle \ref{twistingprinciple} for module Hom-algebras.

\begin{theorem}
\label{thm:twistma}
Let $H = (H,\mu_H,\Delta_H)$ be a bialgebra, $A = (A,\mu_A)$ be an associative algebra, and $\rho \colon H \otimes A \to A$ be an $H$-module-algebra structure on $A$.  Suppose $\alpha_H \colon H \to H$ is a bialgebra morphism and $\alpha_A \colon A \to A$ is an algebra morphism such that
\begin{equation}
\label{alpharhomodalg}
\alpha_A \circ \rho = \rho \circ (\alpha_H \otimes \alpha_A).
\end{equation}
For any integers $n, k \geq 0$, define the map
\[
\rhoalpha^{n,k} = \alpha_A^{2^k} \circ \rho \circ (\alpha_H^n \otimes Id_A) \colon H \otimes A \to A.
\]
Then each $\rhoalpha^{n,k}$ gives $A_\beta$ the structure of an $H_\gamma$-module Hom-algebra, where $\beta = \alpha_A^{2^k}$, $\gamma = \alpha_H^{2^k}$, $A_\beta$ is the Hom-associative algebra $(A,\mu_\beta = \beta \circ \mu_A,\beta)$, and $H_\gamma$ is the Hom-bialgebra $(H,\mu_\gamma = \gamma \circ \mu_H, \Delta_\gamma = \Delta_H \circ \gamma,\gamma)$ in Theorem \ref{thm:hombialg}.
\end{theorem}

\begin{proof}
Using Theorem \ref{thm:twistmha}, it suffices to prove the case $\rhoalpha = \rhoalpha^{0,0}$.  Indeed, the computation \eqref{rhoalphank} shows that $(\rhoalpha^{0,0})^{n,k} = \rhoalpha^{n,k}$.  Moreover, Theorem \ref{thm:twistmodule} already tells us that $A_\beta$ is an $H_\gamma$-module with structure map $\rhoalpha^{n,k}$.  Therefore, it suffices to prove the module Hom-algebra axiom \eqref{mha'} for $\rhoalpha = \rhoalpha^{0,0} = \alpha_A \circ \rho$, which says
\begin{equation}
\label{mharhoalpha}
\rhoalpha \circ (\alpha_H^2 \otimes \mu_\beta) = \mu_\beta \circ \rhoalpha^{\otimes 2} \circ (2~3) \circ (\Delta_\gamma \otimes Id_A \otimes Id_A)
\end{equation}
with $\mu_\beta = \alpha_A \circ \mu_A$ and $\Delta_\gamma = \Delta_H \circ \alpha_H$.  With $B = A^{\otimes 2}$, \eqref{mharhoalpha} is true by the following commutative diagram:
\[
\nicearrow
\xymatrix{
H \otimes B  \ar[rr]^-{\alpha_H \otimes Id_B} \ar[dddd]_-{\alpha_H^2 \otimes \mu_\beta} & & H \otimes B \ar[rr]^-{\chi} \ar[dd]^-{Id_H \otimes \mu_A} & & B \ar[rr]^-{\alpha_A^{\otimes 2}} \ar[dd]^-{\mu_A} & & B \ar[dd]^-{\mu_A}\\
& & & & & & \\
& & H \otimes A \ar[rr]^-{\rho} \ar[ddll]_-{\alpha_H \otimes \alpha_A} & & A \ar[rr]^-{\alpha_A} \ar[dd]^-{\alpha_A} & & A \ar[dd]^-{\alpha_A}\\
& & & & & & \\
H \otimes A \ar[rrrr]^-{\rho} & & & & A \ar[rr]^-{\alpha_A} & & A.
}
\]
Here $\chi = \rho^{\otimes 2} \circ (2~3) \circ (\Delta_H \otimes Id_A \otimes Id_A)$.  The composition along the left and then the bottom edges is the left-hand side in \eqref{mharhoalpha}.  The composition along the top and then the right edges is the right-hand side in \eqref{mharhoalpha}.  The left trapezoid is commutative by definition.  The bottom trapezoid is commutative by the assumption \eqref{alpharhomodalg}.  The top middle square is commutative by the module-algebra axiom \eqref{modalgaxiom}. The top right square is commutative by the assumption that $\alpha_A$ is an algebra morphism.
\end{proof}

To use Theorem \ref{thm:twistma}, we need to check the condition \eqref{alpharhomodalg}.  Using Proposition \ref{prop:twistmodule}, it suffices to check \eqref{alpharhomodalg} on a set of algebra generators of $H$ and a linear basis of $A$.  As an illustration of Theorem \ref{thm:twistma}, let us consider Hom-quantum geometry on the Hom-quantum planes.  We begin with Hom-type analogs of a well-known $\uq$-module-algebra structure on the quantum plane.

\begin{example}[\textbf{Hom-quantum geometry on the Hom-quantum planes, I}]
\label{ex:hqg}
Here we work over the ground field $\bC$.  Assume $q \in \bC$ is not a root of unity.  In this example, we construct classes of $\uq_\al$-module Hom-algebra structures on the Hom-quantum planes $\qpalpha$ by applying Theorem \ref{thm:twistma} to a standard $\uq$-module-algebra structure on the quantum plane $\qp$.  Recall the quantum enveloping algebra $\uq$ and the Hom-quantum enveloping algebra $\uq_\al$ from Example \ref{ex:uqg}, where $\al \colon \uq \to \uq$ is defined by
\begin{equation}
\label{aluqsl2}
\al(K^{\pm 1}) = K^{\pm 1}, \quad
\al(E) = \lambda E, \quad\text{and}\quad
\al(F) = \lambda^{-1}F
\end{equation}
with $\lambda$ any fixed non-zero scalar in $\bC$.  Also recall from Example \ref{ex:hqspace} the quantum plane $\qp = \bk\{x,y\}/(yx-qxy)$ and the Hom-quantum plane $\qpalpha$, where in this example the map $\alpha \colon \qp \to \qp$ is defined by
\begin{equation}
\label{alphaxiqp}
\alpha(x) = \xi x \quad\text{and}\quad
\alpha(y) = \xi \lambda^{-1}y
\end{equation}
with $\xi \in \bC$ any fixed scalar.

Let us first recall the relevant $\uq$-module-algebra structure on the quantum plane $\qp$.  The references here are \cite{kassel} (VII.3) and \cite{msmith}.  This $\uq$-module is defined using the following quantum partial derivatives.  For a monomial $x^my^n \in \qp$, define
\begin{equation}
\label{qpartial}
\dqx(x^my^n) = [m]_q x^{m-1}y^n \quad\text{and}\quad
\dqy(x^my^n) = [n]_qx^my^{n-1},
\end{equation}
where $[n]_q$ is the $q$-integer defined in \eqref{qinteger}.  These quantum partial derivatives extend linearly to all of $\qp$.  Also, for $P = P(x,y) \in \qp$, define
\[
\sigmax(P) = P(qx,y) \quad\text{and}\quad
\sigmay(P) = P(x,qy).
\]
Then there is a $\uq$-module-algebra structure $\rho \colon \uq \otimes \qp \to \qp$ on the quantum plane $\qp$ determined on the algebra generators by
\begin{equation}
\label{uqaction}
\begin{split}
EP &= x (\dqy P), \quad FP = (\dqx P)y,\\
KP &= \sigmax \sigmay^{-1}(P) = P(qx,q^{-1}y), \quad\text{and}\\
K^{-1}P &= \sigmay \sigmax^{-1}(P) = (q^{-1}x,qy).
\end{split}
\end{equation}
We claim that Theorem \ref{thm:twistma} can be applied here, i.e., that the condition
\begin{equation}
\label{alpharhoqp}
\alpha \circ \rho = \rho \circ (\al \otimes \alpha)
\end{equation}
hold.

Using Proposition \ref{prop:twistmodule}, it suffices to check \eqref{alpharhoqp} for the algebra generators $\{E,F,K^{\pm 1}\}$ of $\uq$ and for the monomials $x^my^n \in \qp$.  For the generator $E \in \uq$, we have
\[
E(x^my^n) = x([n]_q x^my^{n-1}) = [n]_q x^{m+1} y^{n-1}.
\]
Therefore, we have
\[
\begin{split}
\alpha(E(x^my^n)) &= [n]_q \xi^{m+n}\lambda^{-(n-1)} x^{m+1} y^{n-1}\\
&= (\lambda E)(\xi^{m+n} \lambda^{-n} x^my^n)\\
&= \al(E)\alpha(x^my^n).
\end{split}
\]
Likewise, for the generator $F \in \uq$, we have
\[
\begin{split}
\alpha(F(x^my^n)) &= \alpha([m]_q x^{m-1}y^{n+1})\\
&= [m]_q \xi^{m+n} \lambda^{-(n+1)} x^{m-1}y^{n+1}\\
&= (\lambda^{-1}F)(\xi^{m+n} \lambda^{-n} x^my^n)\\
&= \al(F)\alpha(x^my^n).
\end{split}
\]
Finally, for the generators $K^{\pm 1} \in \uq$ and $P = P(x,y) \in \qp$, we have
\[
\begin{split}
\alpha(K^{\pm 1}P) &= \alpha(P(q^{\pm 1}x,q^{\mp 1}y))\\
&= P(q^{\pm 1}\xi x,q^{\mp 1}\xi \lambda^{-1}y)\\
&= K^{\pm 1}P(\xi x,\xi \lambda^{-1}y)\\
&= \al(K^{\pm 1})\alpha(P).
\end{split}
\]
We have proved \eqref{alpharhoqp}.

By Theorem \ref{thm:twistma}, for any integers $l, k \geq 0$, the map
\[
\rhoalpha^{l,k} = \alpha^{2^k} \circ \rho \circ (\al^l \otimes Id) \colon \uq \otimes \qp \to \qp
\]
gives the Hom-quantum plane $(\qp)_\beta$ the structure of a $\uq_\gamma$-module Hom-algebra, where $\beta = \alpha^{2^k}$ and $\gamma = \al^{2^k}$.  More precisely, we have
\[
\begin{split}
\rhoalpha^{l,k}(E,x^my^n) &= \lambda^l \alpha^{2^k}(E(x^my^n))\\
&= \lambda^l \alpha^{2^k}([n]_q x^{m+1}y^{n-1})\\
&= [n]_q \xi^{2^k(m+n)} \lambda^{l-2^k(n-1)} x^{m+1}y^{n-1}.
\end{split}
\]
Similar computations show that
\[
\rhoalpha^{l,k}(F,x^my^n) = [m]_q \xi^{2^k(m+n)}\lambda^{-l-2^k(n+1)} x^{m-1}y^{n+1}
\]
and, for $P = P(x,y) \in \qp$,
\[
\rhoalpha^{l,k}(K^{\pm 1},P) = P(q^{\pm 1}\xi^{2^k}x, q^{\mp 1}(\xi\lambda^{-1})^{2^k}y).
\]
In summary, we have constructed an uncountable, four-parameter ($\lambda \in \bC \setminus \{0\}$, $\xi \in \bC$, $l,k \geq 0$) family of $\uq_\gamma$-module Hom-algebra structures on the Hom-quantum planes $(\qp)_\beta$.  We recover the original $\uq$-module-algebra structure $\rho$ \eqref{uqaction} on the quantum plane $\qp$ by taking $\xi = \lambda = 1$.
\qed
\end{example}

\begin{example}[\textbf{Hom-quantum geometry on the Hom-quantum planes, II}]
\label{ex2:hqg}
The $\uq$-module-algebra structure on the quantum plane \eqref{uqaction} considered in Example \ref{ex:hqg} is not the only one.  In this example, we consider Hom-type analogs of one such non-standard $\uq$-module-algebra structure on $\qp$.  The reader is referred to \cite{ds} for a complete classification of $\uq$-module-algebra structures on the quantum plane.  Following the setting of \cite{ds}, we assume that $0 < q < 1$.

According to \cite{ds} (Theorem 9 and Proposition 18), there is a $\uq$-module-algebra structure $\rho \colon \uq \otimes \qp \to \qp$ on $\qp$ determined by
\begin{equation}
\label{uqaction'}
\begin{split}
K^{\pm 1}(x^my^n) &= q^{\pm(m-2n)}x^my^n,\\
E(x^my^n) &= q^{1-n}[n]_q x^m y^{n-1},\quad\text{and}\\
F(x^my^n) &= q^{-m}\frac{q^{2m} - q^{2n}}{q-q^{-1}}x^m y^{n+1},
\end{split}
\end{equation}
where $[n]_q$ is defined in \eqref{qinteger}.  Note that $Ex^m = 0$ because $[0]_q = 0$ and that $F(x^m y^n) = 0$ if and only if $m=n$.

Suppose $\al \colon \uq \to \uq$ is defined as in \eqref{aluqsl2} with $\lambda \in \bC \setminus \{0\}$.  Also let $\alpha \colon \qp \to \qp$ be defined as
\[
\alpha(x) = x \quad \text{and} \quad \alpha(y) = \lambda^{-1}y,
\]
which is \eqref{alphaxiqp} with $\xi = 1$.  We claim that Theorem \ref{thm:twistma} can be applied here, i.e., that the condition
\[
\alpha \circ \rho = \rho \circ (\al \otimes \alpha)
\]
hold.  Using Proposition \ref{prop:twistmodule}, it suffices to check this condition for the algebra generators $\{E,F,K^{\pm 1}\}$ of $\uq$ and for the monomials $x^my^n \in \qp$.  This is obviously true for the generators $K^{\pm 1}$.  For the generator $E$, we have
\[
\begin{split}
\alpha(E(x^my^n)) &= q^{1-n}[n]_q \alpha(x^m y^{n-1})\\
&= q^{1-n}[n]_q\lambda^{-(n-1)} x^m y^{n-1}\\
&= (\lambda E)(\lambda^{-n}x^m y^n)\\
&= \al(E)\alpha(x^my^n).
\end{split}
\]
Likewise, for the generator $F$, we have
\[
\begin{split}
\alpha(F(x^my^n)) &= q^{-m}\frac{q^{2m} - q^{2n}}{q-q^{-1}} \alpha(x^m y^{n+1})\\
&= q^{-m}\frac{q^{2m} - q^{2n}}{q-q^{-1}}\lambda^{-(n+1)} x^m y^{n+1}\\
&= (\lambda^{-1}F)(\lambda^{-n} x^m y^n)\\
&= \al(F)\alpha(x^my^n),
\end{split}
\]
as desired.  Note that $\xi = 1$ in $\alpha$ is necessary for the above computations involving $E$ and $F$.

By Theorem \ref{thm:twistma}, for any integers $l, k \geq 0$, the map
\[
\rhoalpha^{l,k} = \alpha^{2^k} \circ \rho \circ (\al^l \otimes Id) \colon \uq \otimes \qp \to \qp
\]
gives the Hom-quantum plane $(\qp)_\beta$ the structure of a $\uq_\gamma$-module Hom-algebra, where $\beta = \alpha^{2^k}$ and $\gamma = \al^{2^k}$.  More precisely, we have
\[
\begin{split}
\rhoalpha^{l,k}(K^{\pm 1},x^m y^n)
&= q^{\pm(m-2n)}\alpha^{2^k}(x^m y^n)\\
&= q^{\pm(m-2n)} \lambda^{-2^kn} x^m y^n,
\end{split}
\]
\[
\begin{split}
\rhoalpha^{l,k}(E,x^my^n)
&= \lambda^l \alpha^{2^k}(E(x^my^n))\\
&= q^{1-n}[n]_q \lambda^{l-2^k(n-1)} x^m y^{n-1},
\end{split}
\]
and
\[
\begin{split}
\rhoalpha^{l,k}(F,x^my^n)
&= \lambda^{-l}\alpha^{2^k}(F(x^my^n))\\
&= q^{-m}\frac{q^{2m} - q^{2n}}{q-q^{-1}} \lambda^{-l-2^k(n+1)} x^m y^{n+1}.
\end{split}
\]
In summary, we have constructed an uncountable, three-parameter ($\lambda \in \bC \setminus \{0\}$, $l,k \geq 0$) family of $\uq_\gamma$-module Hom-algebra structures on the Hom-quantum planes $(\qp)_\beta$.  We recover the original $\uq$-module-algebra structure $\rho$ \eqref{uqaction'} on the quantum plane $\qp$ by taking $\lambda = 1$.
\qed
\end{example}


\end{document}